\documentclass[twoside]{article}

\usepackage{hyperref}
\usepackage{url,intmacros,graphicx}
\usepackage[numbers,sort]{natbib}
\usepackage[small]{caption}

\def\BibTeX{{\rm B\kern-.05em{\sc i\kern-.025em b}\kern-.08em
    T\kern-.1667em\lower.7ex\hbox{E}\kern-.125emX}}

\newcommand{\email}[1]{\\ \small{\url{#1}} \\}
\newcommand{\institution}[1]{\\ \parbox{3.0in}{\small{#1}}}
\newcommand{\keywords}[1]{\small\textbf{Keywords: }#1}
\newcommand{\AMSsubj}[1]{\noindent\textbf{AMS subject classifications: }#1}
\newcommand\whenaccepted{Submitted: November~30, 2016;
                         Revised: June~26, 2017;
                         Accepted: July 20, 2017.}

\def\qed{\unskip\kern10pt{\unitlength1pt\linethickness{.4pt}\framebox(6,6){}}}

\newtheorem{lemma}{Lemma}[section]
\newtheorem{theorem}{Theorem}[section]
\newenvironment{proof}{{\it Proof:\/}}{\hfill\qed}

\usepackage{amsmath}
\usepackage{color}
\newtheorem{definition}[theorem]{Definition}
\newtheorem{remark}{Remark}

\title{Accurate method of verified computing for solutions of semilinear heat equations\footnote{\whenaccepted}}

\author{Akitoshi Takayasu\thanks{AT was partially supported by JSPS KAKENHI (15K17596).}
\institution{Faculty of Engineering, Information and Systems, University of Tsukuba, Ibaraki 305-8573}
\email{takitoshi@risk.tsukuba.ac.jp}
\and
Makoto Mizuguchi
\institution{{Faculty of Science and Engineering, Waseda University}, Tokyo 169-8555}
\email{makoto.math@fuji.waseda.jp}
\and
Takayuki Kubo\thanks{TK was partially supported by JSPS KAKENHI (15K04946).}
\institution{Institute of Mathematics, University of Tsukuba, Ibaraki 305-0006}
\email{tkubo@math.tsukuba.ac.jp}
\and
Shin'ichi Oishi\thanks{SO was partially supported by CREST, JST.}
\institution{Department of Applied Mathematics, Faculty of Science and Engineering, Waseda University, Tokyo 169-8555}
\email{oishi@waseda.jp}
}

\date{}

\begin{document}
\maketitle
\begin{abstract}
We provide an accurate verification method for solutions of heat equations with a superlinear nonlinearity.
The verification method numerically proves the existence and local uniqueness of the exact solution in a neighborhood of a numerically computed approximate solution.
Our method is based on a fixed-point formulation using the evolution operator, an iterative numerical verification scheme to extend a time interval in which the validity of the solution can be verified, and rearranged error estimates for avoiding the propagation of an overestimate.
As a result, compared with the previous verification method using the analytic semigroup, our method can enclose the solution for a longer time.
Some numerical examples are presented to illustrate the efficiency of our verification method. 
\end{abstract}

\AMSsubj{65G40, 65M15, 35K58}


\section{Introduction}\label{sec:intro}
In this paper, we provide an accurate verification method for solutions of the following semilinear heat equations for $J:=(t_0,t_1]\subset\mathbb{R}$  ($0\le t_0<t_1<\infty$) and $\Omega=(0,1)^d\subset\mathbb{R}^d$ ($d=1,2,3$):
\begin{equation}\label{eqn:monJ}
	\left\{
	\begin{array}{ll}
	\partial_{t} u-\Delta u=u^p&\mathrm{in}~J\times{\Omega},\\
	u(t,x)=0,&t\in J,~x\in\partial{\Omega},\\
	u(t_0,x)=u_0(x),&x\in\Omega,
	\end{array}
	\right.
\end{equation}
where $\partial_{t}=\frac{\partial}{\partial t}$,
$\Delta=\frac{\partial^2}{\partial x_1^2}+\dots+\frac{\partial^d}{\partial x_d^2}$,
$D(\Delta)=H^2(\Omega)\cap H^1_0(\Omega)$,
and $u_0\in L^2(\Omega)$ is a given initial function.
We  also require that the exponent $p$ satisfies $1<p<1+\frac{4}{d}$.

The semilinear heat equation~\eqref{eqn:monJ} appears as a canonical nonlinear extension of the heat equation with a monomial nonlinearity.
In mathematical physics, many models involve nonlinear variants of~\eqref{eqn:monJ}.
For example, solutions of the semilinear heat equation represent the heat distribution inside a solid fuel container~\cite{bib:BE1989}.
The blow-up of a solution indicates the ignition phenomenon of the solid fuel, an application of the semilinear heat equation to combustion theory.

Since the semilinear heat equation~\eqref{eqn:monJ} is typical of nonlinear parabolic partial differential equations, there are numerous analytic results for~\eqref{eqn:monJ}.
For example, in~\cite{bib:kaplan}, the solution of~\eqref{eqn:monJ} exists globally in time for small initial data.
In~\cite{bib:Brezis1996}, the solution to the integral equation
\[
	u(t)=e^{(t-t_0)\Delta}u_0+\int_{t_0}^te^{(t-s)\Delta}u(s)^p \, ds
\]
is called the {\em mild solution}\footnote{The family of operators $\left\{e^{t\Delta}\right\}_{t\ge 0}$ denotes the analytic semigroup generated by $\Delta$. For the development of semigroup theory, see~\cite{bib:pazy,bib:yagi}.} of~\eqref{eqn:monJ}.
Such a mild solution is the solution of~\eqref{eqn:monJ} in the class $C\left(J;L^q(\Omega)\right)$, {whose definition will be given later\footnote{We note that all notation including function spaces used throughout this paper will be given in the last part of Section~\ref{sec:intro}.}}, if $q > q_c$ and $q \ge p$ (resp. $q = q_c$ and $q > p$) hold for $q_c = d(p-1)/2$.
Furthermore, there exists $T > t_0$, which depends on $u_0$, such that the mild solution is a unique solution of~\eqref{eqn:monJ} in the class $C\left([t_0,T);L^q(\Omega)\right)$.
A variety of analytical studies (e.g.,~\cite{bib:cazenave1998,bib:quittner2007} and references therein)
give qualitative properties of the solution.
On the other hand, quantitative properties are difficult to obtain by analytic methods alone. 
That is, it is difficult to show how much initial data is small enough for global existence and 
to determine the explicit value of $T$ for the mild solution to be unique solution of~\eqref{eqn:monJ}.

One might use numerical computations to understand quantitative properties.
The numerical analysis of~\eqref{eqn:monJ} has also been studied for parabolic equations (cf.~\cite{bib:fujita2001,bib:thomee2013}).
However, even if numerically computed results seem to converge to the zero function that is one of the steady states of~\eqref{eqn:monJ}, the results of numerical computations do not prove whether these computed solutions are rigorous global solutions.
Therefore, verified computing for the solution of~\eqref{eqn:monJ} would be helpful to obtain rigorous results with quantitative properties.

Verification methods for parabolic partial differential equations {(\cite{bib:Kinoshita2014,bib:Mizuguchi2014,bib:Mizuguchi2017,bib:Nakao2013,bib:Nakao2012}, etc.)} give mathematical proof of the existence and local uniqueness of the exact solution in a neighborhood of a numerically computed approximate solution.
In such a framework, we try to enclose the solution rigorously in a Banach space $X$.
Namely, let $Y$ be a Banach space with respect to the space variable.
Under the assumption that the approximate solution $\omega$ satisfies $\|u_0-\omega(t_0)\|_Y\le \varepsilon_0$ at $t=t_0$,
we rigorously obtain the enclosure $B_J\left(\omega,\rho_{\varepsilon_0}\right)=\left\{u\in X:\|u-\omega\|_{X}\le\rho_{\varepsilon_0}\right\}$ of the exact solution.
The pioneering work of such a verification method has been studied by M. T. Nakao, T. Kinoshita, and T. Kimura~\cite{bib:Kinoshita2014,bib:Nakao2013,bib:Nakao2012} under the setting of the following Banach spaces: $X=L^2\left(J;H_0^1(\Omega)\right)$ and $Y=H^1_0(\Omega)$.
They only consider the case when the initial function is always the zero function, i.e., $u_0\equiv 0$.
The objective of their work is to derive the norm estimate of the inverse operator of linearized parabolic differential operators.
Subsequently, we have proposed another verification method~\cite{bib:Mizuguchi2014,bib:Mizuguchi2017} based on the analytic semigroup generated by $\Delta$.
This framework gives the enclosure of the mild solution under the setting of Banach spaces: $X=L^\infty\left(J;H_0^1(\Omega)\right)$ and $Y=H^1_0(\Omega)$.
Here, the initial function $u_0$ is any function in $H_0^1(\Omega)$. 
In addition, we have introduced a recursive scheme for enclosing the mild solution in several time intervals.
By using this scheme, we can extend a time interval in which the validity of the solution is verified.

Another recent approach of the verification methods for parabolic partial differential equations has been developed from the viewpoint of dynamical systems.
For example, in~\cite{bib:Gameiro2016,bib:Zgliczynski2004,bib:Zgliczynski2010}, the existence of periodic orbits of the Kuramoto-Sivashinsky equation has been proved by making good use of the spectral method.
Besides that, there is a framework on the verification of invariant objects of parabolic partial differential equations~\cite{bib:Figueras2016}.
To understand structures of dynamical systems, it is natural to require rigorously tracking trajectories of initial value problems of parabolic partial differential equations for long time.
However, there are few studies concerning tracking trajectories of the initial value problems in the field of dynamical systems.

Considering the above background, the goal of this paper is to provide an accurate verification method for parabolic partial differential equations under the setting of the Banach spaces $X = C(J;D(\Delta_\mu^\alpha))$ and $Y=L^2(\Omega)$, where $\Delta_\mu^\alpha$ denotes a fractional power of the shifted positive operator defined in Definition \ref{def:fractionalpower}.
Our framework admits non-zero initial functions satisfying $u_0\in L^2(\Omega)$.
Here, the {\em accurate verification method} refers to a verification method that succeeds in enclosing the exact solution for a long time.
In verified computing of solutions to time evolution equations, when we require the enclosure of the solution for a long time, over-estimates accumulate and prevent us from succeeding in the verification (e.g., see~\cite{bib:Mizuguchi2017}).
In particular, for ordinary differential equations, the propagation of over-estimation is called the \textit{wrapping effect} (e.g., see~\cite{bib:Berz1998,bib:kashikashi,bib:lohner,bib:Nedialkov1999,bib:Zgliczynski2002} and references therein.).
The main contribution of this paper is that a more accurate result than the previous result~\cite{bib:Mizuguchi2017} is obtained for enclosing the mild solution of~\eqref{eqn:monJ}.
The results are shown numerically in Section \ref{sec:numericalexam}.
There are two important points that improve the verification method.
One is the fixed-point formulation using the evolution operator, which is introduced in Section \ref{sec:fixedform}, instead of the analytic semigroup used in~\cite{bib:Mizuguchi2017}.
The other is a rearrangement of the computations for avoiding the propagation of over-estimate by carefully handling the product appearing in the error estimate.

The rest of this paper is organized as follows: Notations used throughout this paper are listed in the rest of this section.
In Section \ref{sec:evop}, the evolution operator proposed by H.~Tanabe~\cite{bib:Tanabe1960} and P. E. Sobolevskii~\cite{bib:Sobolevskii1961} is introduced to transform the initial-boundary value problem~\eqref{eqn:monJ} into a fixed-point form.
Then, the fixed-point form is derived so that the existence of its fixed point and the existence of a mild solution of~\eqref{eqn:monJ} are equivalent.
We also prepare several estimates associated with the evolution operator.
In Section \ref{sec:localinc}, the local inclusion theorem whose sufficient condition can be checked numerically is presented in Theorem \ref{thm:maintheorem}.
Subsequently, in Section \ref{sec:concatescheme}, we provide an iterative numerical verification scheme based on Theorem \ref{thm:maintheorem}.
We also introduce a technique for shrinking the propagation of over-estimates based on techniques
for avoiding the {\em wrapping effect} in verification methods for ordinary differential equations.
Finally, we numerically demonstrate the efficiency of the provided verification method in Section~\ref{sec:numericalexam}.

\subsection*{Notation}
\begin{itemize}
\item [$\mathbb{R}$]: the set of real numbers.
\item [$\mathbb{N}$]: the set of natural numbers.
\item [$\mathbb{C}$]: the set of complex numbers.
\item [$L^q(\Omega)$]: the set of {$q$-th} power Lebesgue integrable functions on $\Omega$ for $q\in[1,\infty)$ with the norm
\[
	\|f\|_{L^q}:=\left(\int_\Omega |f(x)|^q \, dx\right)^{1/q}.
\]
\item [$L^{\infty}(\Omega)$]: the set of essentially bounded functions  on $\Omega$ with the norm
\[
	\|f\|_{L^\infty}:=\mathop{\mathrm{ess\,sup}}_{x\in\Omega}|f(x)|.
\]
\item [$(\phi,\psi)_{L^2}$]: the inner product of $L^2(\Omega)$ defined by
\[
	(\phi,\psi)_{L^2}:=\int_\Omega \phi(x)\psi(x) \, dx.
\]
\item [$\|S\|_{L^q,L^r}$]: the operator norm of $S:L^q(\Omega)\to L^r(\Omega)$ ($1\le q,r\le\infty$) {defined by
\[
	\|S\|_{L^q,L^r}:=\sup_{0\not=\varphi\in L^q(\Omega)}\frac{\|S\varphi\|_{L^r}}{\|\varphi\|_{L^q}}.
\]
}
\item [$H^m(\Omega)$]: the $m$-th order Sobolev space of $L^2(\Omega)$.
\item [$H_0^1(\Omega)$]: the subspace of $H^1(\Omega)$ defined by $\{\phi\in H^1(\Omega):\phi=0~\mbox{on}~\partial\Omega\}$,
where $\phi=0~\mbox{on}~\partial\Omega$ is meant in the trace sense.
{The norm of $H^1_0(\Omega)$ is defined by
\[
\|\phi\|_{H^1_0}:=\left(\|\nabla \phi\|_{L^2}^2+\mu\|\phi\|_{L^2}^2\right)^{1/2}
\]
for $\phi\in H^1_0(\Omega)$ and a certain $\mu > 0$.}
\item [$L^{q}\left(J;Y\right)$]: the time-dependent Lebesgue space as a space of $Y$-valued functions on $J$ with the norm
\[
	\|u\|_{L^{q}\left(J;Y\right)}:=
	\left\{
	\begin{array}{ll}
	\displaystyle\left(\int_J \|u(t,\cdot)\|_Y^q \, dt\right)^{1/q},&1\le q<\infty,\\[4mm]
	\displaystyle\mathop{\mathrm{ess\,sup}}_{t\in J}\|u(t,\cdot)\|_{Y},& q=\infty.
	\end{array}\right.
\]
\item [$C\left(J;Y\right)$]: the time-dependent space as a space of bounded $Y$-valued functions on $J$ with the norm
\[
	\|u\|_{C\left(J;Y\right)}:=\sup_{t\in J}\|u(t,\cdot)\|_{Y}.
\]


\end{itemize}

\section{Evolution Operator}\label{sec:evop}

\subsection{Fixed-point Formulations}\label{sec:fixedform}
For a fixed $\mu>0$, let $\Delta_\mu:=-\Delta+\mu$.
As the inverse of the operator $\Delta_\mu$ is a compact self-adjoint operator, the spectral theorem~\cite{bib:Dautray2000} shows that the operator $\Delta_\mu$ has {a positive discrete spectrum}.
Let $\lambda_A:=\lambda_{\min}+\mu$ be the minimal eigenvalue of $\Delta_\mu$, where $\lambda_{\min}$ denotes the minimal eigenvalue of $-\Delta$.
We set $\sigma>0$ satisfying
\[
	\sigma-p\omega(t,x)^{p-1}\ge\mu,~~t\in J,~a.e.~x\in\Omega,
\]
{where $\omega(t,x)$ is a numerically computed approximate solution of~\eqref{eqn:monJ}.}
For $t\in J$, let $A(t):=-\Delta+\left(\sigma-p\omega(t)^{p-1}\right)$, where $\omega(t)\equiv\omega(t,\cdot)$.
The domain of the operator $A(t)$, denoted by $D(A(t))$, is equal to $D(\Delta)$, and there exists $M > 0$ such that 
\begin{eqnarray*}
	\left|(A(t)\phi,\psi)_{L^2}\right|
	&=&\left|\left(\nabla \phi,\nabla \psi\right)_{L^2}+\left(\sigma-p\omega(t)^{p-1})\phi,\psi\right)_{L^2}\right|\\
	&=&\left|\left(\nabla \phi,\nabla \psi\right)_{L^2}+\mu(\phi,\psi)_{L^2}+\left((\sigma-p\omega(t)^{p-1}-\mu)\phi,\psi\right)_{L^2}\right|\\
	&\le&M\|\phi\|_{H^1_0}\|\psi\|_{H^1_0}
\end{eqnarray*}
for $\phi\in D(A(t))$ and $\psi\in H^1_0(\Omega)$.
It also follows that for $\phi\in D(A(t))$,
\begin{eqnarray}\label{eqn:coercive}
	(A(t)\phi,\phi)_{L^2}&=&\left(\nabla \phi,\nabla \phi\right)_{L^2}+\left((\sigma-p\omega(t)^{p-1})\phi,\phi\right)_{L^2}\nonumber\\
	&\ge&\left(\nabla \phi,\nabla \phi\right)_{L^2}+\mu(\phi,\phi)_{L^2}
	=\|\phi\|_{H^1_0}^2.
\end{eqnarray}
Hence,  $-A(t)$ becomes a sectorial operator\footnote{Let $X$ be a Banach space and $A:D(A)\subset X\to X$ a linear closed operator.
If the resolvent set $\rho(A)$ of $A$ contains a sector $S=\{\lambda\in\mathbb{C}:|\arg(\lambda)|<\theta\}$ with $\theta\in(\pi/2,\pi)$, and there exists $M,>,0$ such that $\|\lambda(\lambda I-A)^{-1}\varphi\|_X\le M\|\varphi\|_X$ for $\lambda\in S$ and $\varphi\in X$,
then $A$ is called a sectorial operator.} on $L^2(\Omega)$.
It follows~\cite{bib:pazy} that $-A(t)$ generates the analytic semigroup $\left\{e^{-sA(t)}\right\}_{s\ge0}$ for each $t\in J$.
Furthermore, one can prove that there exists $C>0$ such that
\[
	\left\|\left(A(t)-A(s)\right)A(s)^{-1}\varphi\right\|_{L^2}\le C|t-s|\|\varphi\|_{L^2},~\forall\varphi\in L^2(\Omega),~t,s\in J,
\]
if the approximate solution is a sufficiently smooth with respect to the $t$-variable.
This is proved from the following facts:
From~\eqref{eqn:coercive} and the embedding $H_0^1(\Omega)\hookrightarrow L^2(\Omega)$, there exists $c>0$ such that $\|A(s)^{-1}\varphi\|_{L^2}\le c\|\varphi\|_{L^2}$ for any $\varphi\in L^2(\Omega)$.
For each $t,s\in J$, the mean-value theorem and Taylor's theorem imply
\begin{align*}
	&\left\|\left(A(t)-A(s)\right)A(s)^{-1}\varphi\right\|_{L^2}\\
	&=p\left\|\left(\omega(t)^{p-1}-\omega(s)^{p-1}\right)A(s)^{-1}\varphi\right\|_{L^2}\\
	&=p(p-1)\left\|\int_0^1\left(\theta\omega(t)+(1-\theta)\omega(s)\right)^{p-2}d\theta\left(\omega(t)-\omega(s)\right)A(s)^{-1}\varphi\right\|_{L^2}\\
	&\le p(p-1)\int_0^1\left\|\left(\theta\omega(t)+(1-\theta)\omega(s)\right)^{p-2}\right\|_{L^\infty}d\theta\cdot\left\|\partial_t\omega(\eta)\right\|_{L^\infty}\left|t-s\right|\left\|A(s)^{-1}\varphi\right\|_{L^2}\\
	&\le C|t-s|\|\varphi\|_{L^2},~\forall\varphi\in L^2(\Omega),
\end{align*}
where $\eta\in J$ and
\[
	C=cp(p-1)\sup_{t,s\in J}\int_0^1\left\|\left(\theta\omega(t)+(1-\theta)\omega(s)\right)^{p-2}\right\|_{L^\infty}d\theta\cdot\sup_{\eta\in J}\left\|\partial_t\omega(\eta)\right\|_{L^\infty}.
\]
From these facts, $-A(t)$ generates the {\em evolution operator} $$\left\{U(t,s)\right\}_{t_0\le s\le t\le t_1}$$ on $L^2(\Omega)$~\cite{bib:pazy,bib:Sobolevskii1961,bib:Tanabe1960,bib:yagi}, etc.
The evolution operator is the solution operator of the homogeneous initial value problem
\begin{equation}\label{eqn:homo}
	\left\{
	\begin{array}{ll}
	\partial_{t} u+A(t) u=0,& t_0\le s<t\le t_1,\\
	u(s)=\varphi,&
	\end{array}
	\right.
\end{equation}
where $\varphi\in L^2(\Omega)$.
It gives the formula $u(t)=U(t,s)\varphi$ for representing a solution of~\eqref{eqn:homo}.
For representing the evolution operator, in the 1960's, {H.~Tanabe~\cite{bib:Tanabe1960} and P.~E.~Sobolevskii~\cite{bib:Sobolevskii1961}} independently constructed the evolution operator when the domain $D(A(t))$ is independent of the variable $t$.

By using the evolution operator $\left\{U(t,s)\right\}_{t_0\le s\le t\le t_1}$ generated by $-A(t)$,
we define a nonlinear operator $\mathcal{S}:C(J;L^2(\Omega))\to C(J;L^2(\Omega))$ as
\begin{equation}\label{eqn:Tv}
	(\mathcal{S}(v))(t):=U(t,t_0)v(t_0)+\int_{t_0}^tU(t,s)g(v(s)) \, ds,
\end{equation}
where 
\[
	g(v):=e^{-\sigma(t-t_0)}\bigg\{\left(\omega+e^{\sigma(t-t_0)}v\right)^p-\omega^p-p\omega^{p-1}e^{\sigma(t-t_0)}v+\omega^p-\partial_{t} \omega+\Delta \omega\bigg\}.
\]
Let $v:=e^{-\sigma(t-t_0)}(u-\omega)$.
The main assertion of this paper is that {\em $u$ is a mild solution of~\eqref{eqn:monJ} if and only if $v$ is a {fixed point} of the operator $\mathcal{S}$ in an appropriate function space}.
In Theorem \ref{thm:maintheorem}, we will give a sufficient condition for guaranteeing the existence and local uniqueness of such a {fixed point}, which can be checked numerically.

\subsection{Estimates Associated with the Evolution Operator}

\begin{lemma}\label{lem:U_es}
For $0\le s<t$, let $\left\{U(t,s)\right\}_{0\le s<t}$ be an evolution operator generated by $-A(t)$.
If there exists a {bound} $\eta$ such that
\[
	\eta\le\inf_{\phi\in D(\Delta)}\frac{\left(A(r)\phi,\phi\right)_{L^2}}{(\phi,\phi)_{L^2}},~\forall r\in[s,t],
\]
then the following estimate holds for $0\le s<t$:
\[
	\left\|U(t,s)\varphi\right\|_{L^2}\le e^{-(t-s)\eta}\|\varphi\|_{L^2},~\forall\varphi\in L^2(\Omega).
\]
\end{lemma}
\begin{proof}
Assume that $u(t)$ is the solution of~\eqref{eqn:homo}.
The energy estimate {\cite{bib:Evans1998}} {implies}
\begin{align*}
\frac{1}{2}\frac{d}{dt}\|u(t)\|_{L^2}^2&=\left(\partial_t u(t),u(t)\right)_{L^2}\\
&=-\left(A(t) u(t),u(t)\right)_{L^2}\\
&\le-\eta\|u(t)\|_{L^2}.
\end{align*}
From the Gronwall inequality {\cite{bib:Gronwall1919}}, it follows that
\[
	\|u(t)\|_{L^2}\le e^{-(t-s)\eta}\|\varphi\|_{L^2}.
\]
\end{proof}

\begin{remark}
When $A(t)=-\Delta+\left(\sigma-p\omega(t)^{p-1}\right)$, 
$\eta$ of Lemma~\ref{lem:U_es} can be taken as $\lambda_A=\lambda_{\min}+\mu$ from~\eqref{eqn:coercive}.
\end{remark}

Unless otherwise noted, we fix $A(t)=-\Delta+\left(\sigma-p\omega(t)^{p-1}\right)$ and $\Delta_\mu=-\Delta+\mu$.
We derive a formula of the evolution operator $\left\{U(t,s)\right\}_{t_0\le s\le t\le t_1}$ by using the analytic semigroup generated by $-A(s)$.
From~\eqref{eqn:homo} we have for any $s$ and $t$ satisfying $t_0 \le s \le t \le t_1$, 
\begin{align*}
	\partial_{t} u+ A(t)u=0\iff\partial_{t} u+ A(s)u=-(A(t)-A(s))u.
\end{align*}
Since $-A(s)$ generates the analytic semigroup $\{e^{-tA(s)}\}_{t\ge0}$, we have
\begin{align*}
	u(t)=&e^{-(t-s)A(s)}\varphi+\int_{s}^te^{-(t-r)A(s)}\{-(A(r)-A(s))u(r)\} \, dr \\
	=&e^{-(t-s)A(s)}\varphi+\int_{s}^te^{-(t-r)A(s)}p\left(\omega(r)^{p-1}-\omega(s)^{p-1}\right)u(r) \, dr.
\end{align*}
That is, the evolution operator satisfies the operator-valued integral equation
\begin{equation}\label{eqn:evop_form}
	U(t,s)=e^{-(t-s)A(s)}+\int_{s}^te^{-(t-s)A(s)}p\left(\omega(r)^{p-1}-\omega(s)^{p-1}\right)U(r,s) \, dr.
\end{equation}
By using~\eqref{eqn:evop_form} in the following, we define the fractional power of $\Delta_\mu$ and introduce several estimates associated with the evolution operator.

\begin{definition}\label{def:fractionalpower}
For $i\in\mathbb{N}$, let $\lambda_i>0$ be the eigenvalue of $\Delta_\mu$.
The function $\phi_i$ denotes an eigenfunction of $\Delta_\mu$ corresponding to $\lambda_i$ satisfying $(\phi_i,\phi_j)_{L^2}=\delta_{ij}$, where $\delta_{ij}$ is Kronecker's delta.
We describe the eigenvalue decomposition of $\varphi\in L^2(\Omega)$ as $\varphi=\sum_{i=1}^\infty c_i\phi_i$, where $c_i=(\varphi,\phi_i)_{L^2}$.
For $\alpha\in(0,1)$, we define the fractional power of $\Delta_\mu$ as
\[
	\Delta_\mu^\alpha \varphi:=\sum_{i=1}^\infty\lambda_i^{\alpha}c_i\phi_i,~D(\Delta_\mu^\alpha):=\left\{\varphi=\sum_{i=1}^\infty c_i\phi_i\in L^2(\Omega):\sum_{i=1}^\infty c_i^2\lambda_i^{2\alpha}<\infty\right\}.
\]
\end{definition}

\begin{lemma}\label{lem:fractional_pow_norm}
For $\alpha\in(0,1)$ and each $t\in J$, {it follows that}
\[
	\left\|\Delta_\mu^\alpha \phi\right\|_{L^2}\le\left\|A(t)^\alpha \phi\right\|_{L^2},~\forall \phi\in D(A(t)).
\]
\end{lemma}
{From~\eqref{eqn:coercive}, it is easy to prove Lemma \ref{lem:fractional_pow_norm}.}

The next lemma is about the embedding $D\left(\Delta_\mu^{\alpha}\right)\hookrightarrow L^p(\Omega)$. It has been shown in several textbooks (e.g.,~\cite{bib:pazy}).

\begin{lemma}\label{lem:frac_emb}
Let $\Omega\subset \mathbb{R}^d~(d\in\mathbb{N})$ be a bounded domain.
Let $\alpha$ satisfy $\alpha>\frac{d(p-2)}{4p}$ for $2<p\leq\infty$, with $\alpha>\frac{d}{4}$ in the case $p=\infty$.
It holds for any $\phi\in D\left(\Delta_\mu^{\alpha}\right)$
\[ 
	\|\phi\|_{L^p}\le C_{p,\alpha}\|\Delta_\mu^{\alpha}\phi\|_{L^2},~~~
	C_{p,\alpha}:=\frac{\Gamma\left(\alpha-\frac{d(p-2)}{4p}\right)}{(4\pi)^{\frac{d(p-2)}{4p}}\Gamma(\alpha)}\min_{0<\beta\leq1}\zeta(\beta),
\]
where $\zeta(\beta)=\beta^{-\frac{d(p-2)}{4p}}((1-\beta)\lambda_{\min}+\mu)^{-\left(\alpha-\frac{d(p-2)}{4p}\right)}$, $\lambda_{\min}$ is the minimal eigenvalue of $-\Delta$, and $\Gamma$ denotes the Gamma function
\[
	\Gamma(x) = \int_{0}^{\infty}t^{x-1}e^{-t} \, dt,~~x>0.
\]
\end{lemma}
The proof of Lemma \ref{lem:frac_emb} is {similar to} the proof in~\cite{bib:Mizuguchi_nolta2016}.\\[2mm]
\begin{proof}
Let $\{e^{-t\Delta_{\mu}}\}_{t\ge 0}$ be the analytic semigroup on $L^2(\Omega)$ generated by $-\Delta_{\mu}$.
For $\beta>0$, it is known that
\begin{equation}\label{gammahan}
\Delta_{\mu}^{-\beta}\varphi=\Gamma(\beta)^{-1}\int_{0}^{\infty}t^{\beta-1}e^{-t\Delta_{\mu}}\varphi \, dt
\end{equation}
holds~\cite{bib:pazy,bib:yagi} for $\varphi\in L^2(\Omega)$.
For $p,q,r>0$ satisfying $1\leq q<p\leq \infty$ and $\frac{1}{r}=\frac{1}{q}-\frac{1}{p}$,
the following estimate associated with the analytic semigroup $\{e^{t\Delta}\}_{t_\ge0}$ generated by
$\Delta$ holds:
\begin{equation}\label{eqn:Lq-es}
\|e^{t\Delta}\varphi\|_{L^p}\leq (4\pi t)^{-\frac{d}{2r}}\|\varphi\|_{L^q},~~~\forall \varphi\in L^q(\Omega),
\end{equation}
where we set $1/\infty=0$.

For {$2<p\leq\infty$}, $\alpha>\frac{d(p-2)}{4p}$, $0<\beta\le1$, and {any $\phi\in D\left(\Delta_{\mu}^{\alpha}\right)$},
\eqref{gammahan} implies
\begin{align*}
\|\phi\|_{L^p}
&=\|\Delta_\mu^{-\alpha}\Delta_\mu^{\alpha}\phi\|_{L^p}\\
&=\left\|\Gamma(\alpha)^{-1}\int_{0}^{\infty}t^{\alpha-1}e^{-t\Delta_{\mu}}\Delta_\mu^{\alpha}\phi \, dt\right\|_{L^p}\\
&\leq\Gamma(\alpha)^{-1}\int_{0}^{\infty}t^{\alpha-1}\|e^{-t\Delta_{\mu}}\Delta_{\mu}^{\alpha}\phi\|_{L^p} \, dt\\
&\leq\Gamma(\alpha)^{-1}\int_{0}^{\infty}t^{\alpha-1}\|e^{-t\Delta_{\mu}}\|_{L^2,L^p}\|\Delta_{\mu}^{\alpha}\phi\|_{L^2} \, dt\\
&=\Gamma(\alpha)^{-1}\int_{0}^{\infty}t^{\alpha-1}\|e^{-\beta t\Delta_{\mu}}e^{-(1-\beta)t\Delta_{\mu}}\|_{L^2,L^p}\|\Delta_{\mu}^{\alpha}\phi\|_{L^2} \, dt\\
&\leq\Gamma(\alpha)^{-1}\int_{0}^{\infty}t^{\alpha-1}\|e^{-\beta t\Delta_{\mu}}\|_{L^2,L^p}\|e^{-(1-\beta)t\Delta_{\mu}}\|_{L^2,L^2}\|\Delta_{\mu}^{\alpha}\phi\|_{L^2} \, dt,
\end{align*}
where we used the semigroup property $e^{-t\Delta_{\mu}}=e^{-\beta t\Delta_{\mu}}e^{-(1-\beta)t\Delta_{\mu}}$ (e.g.,~\cite{bib:pazy}).
One can see that $e^{-t\Delta_{\mu}}=e^{-\mu t} e^{t\Delta}$.
Setting $q=2$ in~\eqref{eqn:Lq-es} and $\frac{d}{2r}=\frac{d(p-2)}{4p}$, we have
\begin{align*}
	&\|\phi\|_{L^p}\\
	&\le\Gamma(\alpha)^{-1}\int_{0}^{\infty}t^{\alpha-1}(4\pi\beta t)^{-\frac{d(p-2)}{4p}}e^{-\beta\mu t}e^{-t(1-\beta)(\lambda_{\min}+\mu)}\|\Delta_\mu^\alpha \phi\|_{L^2} \, dt\\
	&=(4\pi\beta)^{-\frac{d(p-2)}{4p}}\Gamma(\alpha)^{-1}\int_{0}^{\infty}t^{\alpha-1-\frac{d(p-2)}{4p}}e^{-t((1-\beta)\lambda_{\min}+\mu)} \, dt~\|\Delta_\mu^\alpha \phi\|_{L^2}\\
	&=(4\pi\beta)^{-\frac{d(p-2)}{4p}}\Gamma(\alpha)^{-1}\left(\frac{1}{(1-\beta)\lambda_{\min}+\mu}\right)^{\alpha-\frac{d(p-2)}{4p}}\Gamma\left(\alpha-\frac{d(p-2)}{4p}\right)\|\Delta_\mu^\alpha \phi\|_{L^2}\\
	&=\frac{\Gamma\left(\alpha-\frac{d(p-2)}{4p}\right)}{(4\pi)^{\frac{d(p-2)}{4p}}\Gamma(\alpha)}\beta^{-\frac{d(p-2)}{4p}}((1-\beta)\lambda_{\min}+\mu)^{-\left(\alpha-\frac{d(p-2)}{4p}\right)}\|\Delta_\mu^\alpha \phi\|_{L^2}
\end{align*}
for any $\phi\in D\left(\Delta_{\mu}^{\alpha}\right)$.
Hence, setting $\zeta(\beta)=\beta^{-\frac{d(p-2)}{4p}}((1-\beta)\lambda_{\min}+\mu)^{-\left(\alpha-\frac{d(p-2)}{4p}\right)}$,
\[
	\|\phi\|_{L^p}\leq\frac{\Gamma\left(\alpha-\frac{d(p-2)}{4p}\right)}{(4\pi)^{\frac{d(p-2)}{4p}}\Gamma(\alpha)}\min_{0<\beta\leq1}\zeta(\beta)\|\Delta_\mu^\alpha \phi\|_{L^2},~\forall\phi\in D(\Delta_{\mu}^{\alpha}).
\]
\end{proof}


\begin{lemma}\label{lem:semigroup_es2}
For each $t\in J$, let $\left\{e^{-sA(t)}\right\}_{s\ge0}$ be an analytic semigroup generated by $-A(t)$.
For $\alpha\in(0,1)$ and $\beta\in(0,1]$, it follows that for any $\varphi\in L^2(\Omega)$
\begin{align*}
	\left\|A(t)^{\alpha}e^{-sA(t)}\varphi\right\|_{L^2}\le\left(\frac{\alpha}{e\beta}\right)^{\alpha}s^{-\alpha}e^{-s(1-\beta)\lambda_A}\|\varphi\|_{L^2}.
\end{align*}
\end{lemma}
\begin{proof}
Since all the values from the spectrum of $A(t)$ are positive, the spectral mapping theorem implies
\begin{align*}
	\left\|A(t)^{\alpha}e^{-sA(t)}\varphi\right\|_{L^2}&=\sup_{x\in\sigma(A(t))}\left(x^\alpha e^{-sx}\right)\|\varphi\|_{L^2}\\
	&\le\sup_{x\in\sigma(A(t))}\left(x^\alpha e^{-s\beta x}\right)\sup_{x\in\sigma(A(t))}\left(e^{-s(1-\beta)x}\right)\|\varphi\|_{L^2}\\
	&\le \left(\frac{\alpha}{e\beta}\right)^{\alpha}s^{-\alpha}e^{-s(1-\beta)\lambda_A}\|\varphi\|_{L^2},~\forall \varphi\in L^2(\Omega),
\end{align*}
where $\sigma(A(t))$ denotes the spectrum of $A(t)$, and we used the inequalities
\[
	\sup_{x\in\sigma(A(t))}\left(x^\alpha e^{-s\beta x}\right)\le\left(\frac{\alpha}{e\beta}\right)^{\alpha}s^{-\alpha}~\mbox{and}~\sup_{x\in\sigma(A(t))}\left(e^{-s(1-\beta)x}\right)\le e^{-s(1-\beta)\lambda_A}.
\]
\end{proof}

\begin{lemma}\label{lem:U_ts_es}
For $\alpha\in (0,1)$, $\beta\in(0,1]$, $s$ and $t \in J$ satisfying $s \le t$, and $C_\omega$ a constant satisfying
\begin{equation}\label{eqn:C_omega}
	\left\|p\left(\omega(t)^{p-1}-\omega(s)^{p-1}\right)\varphi\right\|_{L^2}\le C_\omega(t-s)\|\varphi\|_{L^2},~\forall\varphi\in L^2(\Omega),
\end{equation}
it follows that for any $\varphi\in L^2(\Omega)$,
\begin{eqnarray*}
	\left\|\Delta_\mu^\alpha U(t,s)\varphi\right\|_{L^2}
	&\le&\left(\frac{\alpha}{e\beta}\right)^\alpha(t-s)^{-\alpha} e^{-(t-s)(1-\beta)\lambda_A}\left\{1+\frac{C_\omega(t-s)^2}{(1-\alpha)(2-\alpha)}\right\}\|\varphi\|_{L^2}.
\end{eqnarray*}
\end{lemma}
\begin{proof}
For $\alpha\in (0,1)$, $\beta\in(0,1]$, and $s$ and $t \in J$, Lemmas \ref{lem:U_es}, \ref{lem:fractional_pow_norm}, \ref{lem:semigroup_es2}, and~\eqref{eqn:C_omega} yield for any $\varphi\in L^2(\Omega)$,
\begin{eqnarray*}
	&&\hspace{-30pt}\left\|\Delta_\mu^\alpha U(t,s)\varphi\right\|_{L^2}\\
	&\le&\left\|\Delta_\mu^\alpha\,e^{-(t-s)A(s)}\varphi\right\|_{L^2}+\int_{s}^t\left\|\Delta_\mu^\alpha\,e^{-(t-r)A(s)}\{-(A(r)-A(s))\}U(r,s)\varphi\right\|_{L^2}dr\\
	&\le&\left\|A(s)^\alpha\,e^{-(t-s)A(s)}\varphi\right\|_{L^2}+\int_{s}^t\left\|A(s)^\alpha\,e^{-(t-r)A(s)}\{-(A(r)-A(s))\}U(r,s)\varphi\right\|_{L^2} \hspace{-7pt} dr  \\
	&\le&\left(\frac{\alpha}{e\beta}\right)^\alpha(t-s)^{-\alpha}e^{-(t-s)(1-\beta)\lambda_A}\|\varphi\|_{L^2}\\
	&&\quad +\int_s^t\left(\frac{\alpha}{e\beta}\right)^\alpha(t-r)^{-\alpha} e^{-(t-r)(1-\beta)\lambda_A}\|(A(r)-A(s))U(r,s)\varphi\|_{L^2} \, dr  \\
	&\le&\left(\frac{\alpha}{e\beta}\right)^\alpha(t-s)^{-\alpha}e^{-(t-s)(1-\beta)\lambda_A}\|\varphi\|_{L^2}\\
	&&\quad +\int_s^t\left(\frac{\alpha}{e\beta}\right)^\alpha(t-r)^{-\alpha} e^{-(t-r)(1-\beta)\lambda_A}C_{\omega}(r-s)\|U(r,s)\varphi\|_{L^2} \, dr\\
	&\le&\left(\frac{\alpha}{e\beta}\right)^\alpha(t-s)^{-\alpha}e^{-(t-s)(1-\beta)\lambda_A}\|\varphi\|_{L^2}\\
	&&\quad +\int_s^t\left(\frac{\alpha}{e\beta}\right)^\alpha(t-r)^{-\alpha} e^{-(t-r)(1-\beta)\lambda_A}C_{\omega}(r-s)e^{-(r-s)\lambda_A}\|\varphi\|_{L^2} \, dr\\
	&=&\left(\frac{\alpha}{e\beta}\right)^\alpha(t-s)^{-\alpha}e^{-(t-s)(1-\beta)\lambda_A}\|\varphi\|_{L^2}\\
	&&\quad +\left(\frac{\alpha}{e\beta}\right)^\alpha C_{\omega}\|\varphi\|_{L^2}e^{-(t-s)(1-\beta)\lambda_A}\int_s^t(t-r)^{-\alpha}(r-s)e^{-(r-s)\beta\lambda_A} \, dr\\
	&\le&\left(\frac{\alpha}{e\beta}\right)^\alpha(t-s)^{-\alpha}e^{-(t-s)(1-\beta)\lambda_A}\|\varphi\|_{L^2}\\
	&&\quad +\left(\frac{\alpha}{e\beta}\right)^\alpha C_{\omega}\|\varphi\|_{L^2}e^{-(t-s)(1-\beta)\lambda_A}(t-s)^{2-\alpha}\frac{\Gamma(1-\alpha)}{\Gamma(3-\alpha)}\\
	&=&\left(\frac{\alpha}{e\beta}\right)^\alpha(t-s)^{-\alpha}e^{-(t-s)(1-\beta)\lambda_A}\left\{1+\frac{C_{\omega}(t-s)^{2}}{(1-\alpha)(2-\alpha)}\right\}\|\varphi\|_{L^2}.
\end{eqnarray*}
\end{proof}

\section{Local Inclusion}\label{sec:localinc}

In this section, we fix the fractional power $\alpha$ in the interval $\left(\frac{d(p-1)}{4p},\frac{1}{p}\right)$,
and we define a weighted subspace of $C(J;D(\Delta_\mu^\alpha))$ as
\[
	X_{\alpha}:=\left\{u\in C(J;D(\Delta_\mu^\alpha)):\sup_{t\in J}(t-t_0)^\alpha\|\Delta_\mu^\alpha u\|_{L^2}<+\infty\right\}
\]
with the norm  $\|u\|_{X_\alpha}:=\sup_{t\in J}(t-t_0)^\alpha\|\Delta_\mu^\alpha u\|_{L^2}$.
Under this norm, $X_\alpha$ becomes a Banach space\footnote{The norm $\|u\|_{X_\alpha}$ is equivalent to the graph norm of $C(J;D(\Delta_\mu^\alpha))$ because of the embedding $D(\Delta_\mu^\alpha)\hookrightarrow L^2(\Omega)$.}.
In addition, we define a neighborhood of $\omega$ as 
\[
	B_J(\omega,\rho):=\left\{u\in C(J;D(\Delta_\mu^\alpha)):\sup_{t\in J}(t-t_0)^\alpha e^{-\sigma(t-t_0)}\|\Delta_\mu^\alpha(u-\omega)\|_{L^2}\le\rho\right\}.
\]
The following theorem gives a sufficient condition for guaranteeing the existence and local uniqueness of the mild solution of~\eqref{eqn:monJ} in $B_J(\omega,\rho)$.

\begin{theorem}\label{thm:maintheorem}
Let $\alpha$ satisfy $\alpha\in\left(\frac{d(p-1)}{4p},\frac{1}{p}\right)$, $p\ge2$, and $\tau=t_1-t_0$.
We assume that the approximate solution $\omega$ satisfies $\|u_0-\omega(t_0)\|_{L^2}\le\varepsilon_{0}$ and
\begin{equation}\label{eqn:residual}
	\left\|\partial_t\omega-\Delta\omega-\omega^p\right\|_{C\left(J;L^2(\Omega)\right)}\le\delta.
\end{equation}
If 
\begin{equation}\label{eqn:cond}
	W(\tau)\left(\varepsilon_0+L_\omega(\rho)\rho^2+\frac{\delta\tau}{1-\alpha}\right)<\rho
\end{equation}
holds for $\rho>0$, then the mild solution $u$ of~\eqref{eqn:monJ} {exists and is unique} in $B_{J}(\omega,\rho)$.
Here, $W(\tau)$ and $L_\omega(\rho)$ are given by
\[
	W(\tau)=\left(\frac{\alpha}{e}\right)^\alpha\left\{ 1+\frac{C_\omega\tau^2}{(1-\alpha)(2-\alpha)}\right\}
\]
and
\[
	L_\omega(\rho)=p(p-1)C_{2p,\alpha}^2e^{\sigma\tau}\left(\tau^\alpha\left\|\omega\right\|_{C(J;L^{2p}(\Omega))}+C_{2p,\alpha}e^{\sigma\tau}\rho\right)^{p-2}\tau^{1-p\alpha}B(1-\alpha,1-p\alpha),
\]
respectively.
The constant $C_\omega$ satisfies~\eqref{eqn:C_omega},
$C_{2p,\alpha}$ is the embedding constant introduced in Lemma $\ref{lem:frac_emb}$, and $B(x,y)$ is the Beta function.
\end{theorem}
The proof of Theorem \ref{thm:maintheorem} is based on Banach's fixed-point theorem using the nonlinear operator $\mathcal{S}$ defined in~\eqref{eqn:Tv}.
Before proving the theorem, we prepare a lemma with respect to the nonlinear term of~\eqref{eqn:monJ}.

\begin{lemma}\label{lem:Lipschitz}
Let $z_1$ and $z_2\ in D\left(\Delta_\mu^\alpha\right)$,
and let $\omega$ be the approximate solution.
For a fixed $t\in J$ and $p\ge2$, it follows that
\begin{eqnarray*}
	&&\hspace{-30pt}\left\|\left(\omega+z_1\right)^p-\left(\omega+z_2\right)^p-p\omega^{p-1}(z_1-z_2)\right\|_{L^2}\\
	&\le&p(p-1)C_{2p,\alpha}^2\int_0^1\int_0^1\left(\|\omega\|_{L^{2p}}+\eta C_{2p,\alpha}\left\|\Delta_\mu^\alpha(\theta z_1+(1-\theta)z_2)\right\|_{L^2}\right)^{p-2}d\eta\\
	&&\quad\left\|\Delta_\mu^\alpha(\theta z_1+(1-\theta)z_2)\right\|_{L^2}d\theta\left\|\Delta_\mu^\alpha (z_1-z_2)\right\|_{L^2},
\end{eqnarray*}
where $C_{2p,\alpha}$ is the embedding constant introduced in Lemma $\ref{lem:frac_emb}$.
\end{lemma}
\begin{proof}
The mean-value theorem {implies} that
\begin{eqnarray*}
	&&\hspace{-25pt}\left(\omega+z_1\right)^p-\left(\omega+z_2\right)^p-p\omega^{p-1}(z_1-z_2)\\
	&&=\int_0^1\frac{d}{d\theta}\left\{\left(\omega+\theta z_1+(1-\theta)z_2\right)^p-\theta p\omega^{p-1}(z_1-z_2)\right\} \, d\theta\\
	&&=p\int_0^1\left(\left(\omega+\theta z_1+(1-\theta)z_2\right)^{p-1}-\omega^{p-1}\right) \, d\theta~(z_1-z_2)\\
	&&=p\int_0^1\int_0^1\frac{d}{d\eta}\left(\omega+\eta(\theta z_1+(1-\theta)z_2)\right)^{p-1} \, d\eta \, d\theta~(z_1-z_2)\\
	&&=p(p-1)\int_0^1\int_0^1\left(\omega+\eta(\theta z_1+(1-\theta)z_2)\right)^{p-2}d\eta(\theta z_1+(1-\theta)z_2) d\theta(z_1-z_2).
\end{eqnarray*}
From H\"older's inequality~\cite{bib:Adams1975}, Minkowski's inequality~\cite{bib:Adams1975}, and Lemma~\ref{lem:frac_emb}, we have
\begin{eqnarray*}
	&&\hspace{-25pt}\left\|\left(\omega+z_1\right)^p-\left(\omega+z_2\right)^p-p\omega^{p-1}(z_1-z_2)\right\|_{L^2}\\
	&& \le p(p-1)\int_0^1\int_0^1\left\|\omega+\eta(\theta z_1+(1-\theta)z_2)\right\|_{L^{2p}}^{p-2}d\eta  \\
	    &&\qquad \left\| \theta z_1 + (1-\theta)z_2\right\|_{L^{2p}}d\theta\left\|z_1-z_2\right\|_{L^{2p}}\\
	&& \le p(p-1)\int_0^1\int_0^1\left(\|\omega\|_{L^{2p}}+\eta\|\theta z_1+(1-\theta)z_2\|_{L^{2p}}\right)^{p-2}d\eta  \\
	    &&\qquad \left\|\theta z_1+(1-\theta)z_2\right\|_{L^{2p}}d\theta\left\|z_1-z_2\right\|_{L^{2p}}\\
	&& \le p(p-1)C_{2p,\alpha}^2\int_0^1\int_0^1\left(\|\omega\|_{L^{2p}}+\eta C_{2p,\alpha}\left\|\Delta_\mu^\alpha(\theta z_1+(1-\theta)z_2)\right\|_{L^2}\right)^{p-2}d\eta\\
	&&\qquad\left\|\Delta_\mu^\alpha(\theta z_1+(1-\theta)z_2)\right\|_{L^2}d\theta\left\|\Delta_\mu^\alpha (z_1-z_2)\right\|_{L^2}.
\end{eqnarray*}
\end{proof}

\medskip\noindent
\begin{proof}\,(Theorem~\ref{thm:maintheorem})
For $\rho > 0$, let $V:=\left\{v\in X_\alpha:\|v\|_{X_\alpha}\le\rho\right\}$.
Consider the nonlinear operator $\mathcal{S}$ defined in~\eqref{eqn:Tv}.
From Banach's fixed-point theorem, we derive a sufficient condition for $\mathcal{S}$ to have a fixed point in $V$.
For $v\in V$ with a fixed $t\in J$,
\begin{equation*}\label{eqn:Fform_estimate}
	\left\|\Delta_\mu^\alpha \mathcal{S}(v(t))\right\|_{L^2}\le\left\|\Delta_\mu^\alpha  U(t,t_0)v(t_0)\right\|_{L^2}+\int_{t_0}^t\left\|\Delta_\mu^\alpha U(t,s)g(v(s))\right\|_{L^2} \, ds.
\end{equation*}
Setting $v_0=v(t_0)$, we have the following estimate from Lemma \ref{lem:U_ts_es} with $\beta=1$:
\begin{align}\label{eqn:T_v_es}\nonumber
	&\hspace{-40pt}\left\|\Delta_\mu^\alpha \mathcal{S}(v(t))\right\|_{L^2}\\
	\le&\left\|\Delta_\mu^\alpha  U(t,t_0)v_0\right\|_{L^2}+\int_{t_0}^t\left\|\Delta_\mu^\alpha U(t,s)g(v(s))\right\|_{L^2} \, ds\nonumber\\
	\le&(t-t_0)^{-\alpha}\left(\frac{\alpha}{e}\right)^\alpha\|v_0\|_{L^2}\left\{1+\frac{C_\omega(t-t_0)^2}{(1-\alpha)(2-\alpha)}\right\}\nonumber\\
	&\qquad +\int_{t_0}^t(t-s)^{-\alpha}\left(\frac{\alpha}{e}\right)^\alpha\|g(v(s))\|_{L^2}\left\{1+\frac{C_\omega(t-s)^2}{(1-\alpha)(2-\alpha)}\right\} \, ds\nonumber\\
	\le&(t-t_0)^{-\alpha}\left(\frac{\alpha}{e}\right)^\alpha\left\{1+\frac{C_\omega(t-t_0)^2}{(1-\alpha)(2-\alpha)}\right\} \nonumber \\
	&\qquad \left(\|v_0\|_{L^2}+(t-t_0)^{\alpha}\int_{t_0}^t(t-s)^{-\alpha}\|g(v(s))\|_{L^2} \, ds\right). 
\end{align}
To estimate $g(v(s))$ in the last term in~\eqref{eqn:T_v_es}, we decompose $g$ into two parts, as $g(v(s))=g_1(s)+g_2(s)$, where
\[
	g_1(s):=e^{-\sigma(s-t_0)}\left\{\left(\omega+e^{\sigma(s-t_0)}v\right)^p-\omega^p-p\omega^{p-1}e^{\sigma(s-t_0)}v\right\},
\]
and
\[
	g_2(s):=e^{-\sigma(s-t_0)}\left(\omega^p-\partial_{t} \omega+\Delta \omega\right).
\]
Setting $z_1=e^{\sigma(s-t_0)}v$ and $z_2=0$ in Lemma \ref{lem:Lipschitz},
the estimate of $g_1(s)$ is
\begin{align*}
	\left\|g_1(s)\right\|_{L^2}
	\le & ~p(p-1)C_{2p,\alpha}^2e^{\sigma(s-t_0)}\left\|\Delta_\mu^\alpha v\right\|_{L^2}^2 \\
	& \qquad \int_0^1\int_0^1\left(\left\|\omega\right\|_{L^{2p}}+\theta\eta C_{2p,\alpha}e^{\sigma(s-t_0)}\left\|\Delta_\mu^\alpha v\right\|_{L^2}\right)^{p-2} \, d\eta\theta \, d\theta.
\end{align*}
It follows that
\begin{align}\label{eqn:g1_es}\nonumber
	&\hspace{-30pt}(t-t_0)^\alpha\int_{t_0}^t(t-s)^{-\alpha}\left\|g_1(s)\right\|_{L^2} \, ds\\\nonumber
	\le&~(t-t_0)^\alpha\int_{t_0}^t(t-s)^{-\alpha}\Bigg\{p(p-1)C_{2p,\alpha}^2e^{\sigma(s-t_0)}\left\|\Delta_\mu^\alpha v\right\|_{L^2}^2\\\nonumber
 	    &\quad\int_0^1\int_0^1\left(\left\|\omega\right\|_{L^{2p}}+\theta\eta C_{2p,\alpha}e^{\sigma(s-t_0)}\left\|\Delta_\mu^\alpha v\right\|_{L^2}\right)^{p-2} \, d\eta\theta \, d\theta\Bigg\} \, ds\\\nonumber
	\le&~(t-t_0)^\alpha p(p-1)C_{2p,\alpha}^2e^{\sigma(t-t_0)}\rho^2 \\\nonumber
	    & \quad \int_0^1\int_0^1\left((t-t_0)^\alpha\left\|\omega\right\|_{C(J;L^{2p}(\Omega))}+\theta\eta C_{2p,\alpha}e^{\sigma(t-t_0)}\rho\right)^{p-2} \, d\eta\theta \, d\theta\\\nonumber
	    &\quad \int_{t_0}^t(t-s)^{-\alpha}(s-t_0)^{-p\alpha} \, ds\\\nonumber
	=&~p(p-1)C_{2p,\alpha}^2e^{\sigma(t-t_0)}\rho^2 \\\nonumber
	&\quad \int_0^1\int_0^1\left((t-t_0)^\alpha\left\|\omega\right\|_{C(J;L^{2p}(\Omega))}+\theta\eta C_{2p,\alpha}e^{\sigma(t-t_0)}\rho\right)^{p-2} \, d\eta\theta \, d\theta\\
	&\quad(t-t_0)^{1-p\alpha}B(1-\alpha,1-p\alpha).
\end{align}
From~\eqref{eqn:residual}, the estimate with respect to $g_2(s)$ is
\begin{align}\label{eqn:g2_es}
	& \hspace{-40pt}(t-t_0)^\alpha\int_{t_0}^t(t-s)^{-\alpha}\left\|g_2(s)\right\|_{L^2} \, ds \nonumber \\
	&\le\delta(t-t_0)^\alpha\int_{t_0}^t(t-s)^{-\alpha}e^{-\sigma(t-t_0)} \, ds
	\le\frac{\delta(t-t_0)}{1-\alpha}.
\end{align}
Hence, \eqref{eqn:T_v_es}, \eqref{eqn:g1_es}, and \eqref{eqn:g2_es} yield
\begin{align}\nonumber
	&\hspace{-10pt}\left\|\mathcal{S}(v)\right\|_{X_\alpha}\\
	\le&~\sup_{t\in J}\left(\frac{\alpha}{e}\right)^\alpha\left\{1+\frac{C_\omega(t-t_0)^2}{(1-\alpha)(2-\alpha)}\right\} 
	\left(\|v_0\|_{L^2}+(t-t_0)^{\alpha}\int_{t_0}^t(t-s)^{-\alpha}\|g(v(s))\|_{L^2} \, ds\right)\nonumber\\
	\le&~\sup_{t\in J}\left(\frac{\alpha}{e}\right)^\alpha\left\{1+\frac{C_\omega(t-t_0)^2}{(1-\alpha)(2-\alpha)}\right\}\nonumber\\
	    &\quad\left(\varepsilon_0+(t-t_0)^{\alpha}\int_{t_0}^t(t-s)^{-\alpha}\|g_1(s)\|_{L^2} \, ds+(t-t_0)^{\alpha}\int_{t_0}^t(t-s)^{-\alpha}\|g_2(s)\|_{L^2} \, ds\right)\nonumber\\
	\le&~\sup_{t\in J}\left(\frac{\alpha}{e}\right)^\alpha\left\{1+\frac{C_\omega(t-t_0)^2}{(1-\alpha)(2-\alpha)}\right\}\Bigg(\varepsilon_0+p(p-1)C_{2p,\alpha}^2e^{\sigma(t-t_0)}\rho^2\nonumber\\
	    &\quad \int_0^1\int_0^1\left((t-t_0)^\alpha\left\|\omega\right\|_{C(J;L^{2p}(\Omega))}+\theta\eta C_{2p,\alpha}e^{\sigma(t-t_0)}\rho\right)^{p-2} \, d\eta\theta \, d\theta \nonumber \\
	    & \qquad(t-t_0)^{1-p\alpha}B(1-\alpha,1-p\alpha) +\frac{\delta(t-t_0)}{1-\alpha}\Bigg)\nonumber
\end{align}
\begin{align}\label{eqn:S(v)inV}
	\le&~\left(\frac{\alpha}{e}\right)^\alpha\left\{1+\frac{C_\omega\tau^2}{(1-\alpha)(2-\alpha)}\right\}\Bigg(\varepsilon_0+p(p-1)C_{2p,\alpha}^2e^{\sigma\tau}\rho^2\nonumber\\
	    &\quad \left(\tau^\alpha\left\|\omega\right\|_{C(J;L^{2p}(\Omega))}+C_{2p,\alpha}e^{\sigma\tau}\rho\right)^{p-2}\tau^{1-p\alpha}B(1-\alpha,1-p\alpha)+\frac{\delta\tau}{1-\alpha}\Bigg)\nonumber\\
	=&~W(\tau)\left(\varepsilon_0+L_\omega(\rho)\rho^2+\frac{\delta\tau}{1-\alpha}\right).
\end{align}
From~\eqref{eqn:cond} and~\eqref{eqn:S(v)inV}, $\left\|\mathcal{S}(v)\right\|_{X_\alpha}<\rho$ holds, implying $\mathcal{S}(v)\in V$.

Next, we show that $\mathcal{S}$ is a contraction mapping on $V$.
Let $v_i\in V$ ($i=1,2$).
From Lemma \ref{lem:U_ts_es},
\begin{align}\label{eqn:T_contraction}\nonumber
	&\hspace{-30pt}\left\|\Delta_\mu^\alpha \left(\mathcal{S}(v_1(t))-\mathcal{S}(v_2(t))\right)\right\|_{L^2}\\
	\le&~\int_{t_0}^t\left\|\Delta_\mu^\alpha U(t,s)\left(g(v_1(s))-g(v_2(s))\right)\right\|_{L^2} \, ds\nonumber\\
	\le&~\left(\frac{\alpha}{e}\right)^\alpha\left\{1+\frac{C_\omega(t-t_0)^2}{(1-\alpha)(2-\alpha)}\right\}\int_{t_0}^t(t-s)^{-\alpha}\|g(v_1(s))-g(v_2(s))\|_{L^2} \, ds.
\end{align}
From
\begin{align*}
	&g(v_1(s))-g(v_2(s))\\
	&=e^{-\sigma(s-t_0)}\left\{\left(\omega+e^{\sigma(s-t_0)}v_1\right)^p-\left(\omega+e^{\sigma(s-t_0)}v_2\right)^p-p\omega^{p-1}e^{\sigma(s-t_0)}(v_1-v_2)\right\},
\end{align*}
setting $z_i=e^{\sigma(s-t_0)}v_i$ ($i=1,2$) in Lemma \ref{lem:Lipschitz}, we have 
\begin{align*}
	&\hspace{-30pt}\left\|g(v_1(s))-g(v_2(s))\right\|_{L^2}\\
	\le&~p(p-1)C_{2p,\alpha}^2e^{\sigma(s-t_0)} \\
	& \quad\int_0^1\int_0^1\left(\|\omega\|_{L^{2p}}+\eta C_{2p,\alpha}e^{\sigma(s-t_0)}\left\|\Delta_\mu^\alpha(\theta v_1+(1-\theta)v_2)\right\|_{L^2}\right)^{p-2} \, d\eta\\
	&\quad\left\|\Delta_\mu^\alpha(\theta v_1+(1-\theta)v_2)\right\|_{L^2}d\theta\left\|\Delta_\mu^\alpha (v_1-v_2)\right\|_{L^2}.
\end{align*}
Since the ball $V$ is convex, and $\|v_i\|_{X_\alpha}\le\rho$ ($i=1,2$), $\|\theta v_1+(1-\theta)v_2\|_{X_\alpha}\le\rho$ holds for $\theta\in[0,1]$.
We then obtain
\begin{align}\label{eqn:g1_g2_es}\nonumber
	&\hspace{-20pt}\int_{t_0}^t(t-s)^{-\alpha}\left\|g(v_1(s))-g(v_2(s))\right\|_{L^2} \, ds\\\nonumber
	\le&~\int_{t_0}^t(t-s)^{-\alpha}\Bigg\{p(p-1)C_{2p,\alpha}^2e^{\sigma(s-t_0)}\\\nonumber
	&\quad\int_0^1\int_0^1\left(\|\omega\|_{L^{2p}}+\eta C_{2p,\alpha}e^{\sigma(s-t_0)}\left\|\Delta_\mu^\alpha(\theta v_1+(1-\theta)v_2)\right\|_{L^2}\right)^{p-2}d\eta\\\nonumber
	&\quad\left\|\Delta_\mu^\alpha(\theta v_1+(1-\theta)v_2)\right\|_{L^2}d\theta\left\|\Delta_\mu^\alpha (v_1-v_2)\right\|_{L^2}\Bigg\} \, ds\\\nonumber
	\le&~p(p-1)C_{2p,\alpha}^2e^{\sigma(t-t_0)}\rho\int_0^1\left((t-t_0)^\alpha\left\|\omega\right\|_{C(J;L^{2p}(\Omega))}+\eta C_{2p,\alpha}e^{\sigma(t-t_0)}\rho\right)^{p-2} \, d\eta\\
	&\quad(t-t_0)^{1-(p+1)\alpha}B(1-\alpha,1-p\alpha)\left\|v_1-v_2\right\|_{X_\alpha}.
\end{align}
Hence, from~\eqref{eqn:T_contraction} and~\eqref{eqn:g1_g2_es}, we have
\begin{align*}\label{eqn:S(v1)_S(v2)}\nonumber
	&\hspace{-20pt}\left\|\mathcal{S}(v_1)-\mathcal{S}(v_2)\right\|_{X_\alpha}\\
	\le&~\sup_{t\in J}(t-t_0)^\alpha\left(\frac{\alpha}{e}\right)^\alpha\left\{1+\frac{C_\omega(t-t_0)^2}{(1-\alpha)(2-\alpha)}\right\} \nonumber \\
	    & \quad	\int_{t_0}^t(t-s)^{-\alpha}\|g(v_1(s))-g(v_2(s))\|_{L^2} \, ds\nonumber\\
	\le&~\sup_{t\in J}(t-t_0)^\alpha\left(\frac{\alpha}{e}\right)^\alpha\left\{1+\frac{C_\omega(t-t_0)^2}{(1-\alpha)(2-\alpha)}\right\}\nonumber\\
    	&\quad p(p-1)C_{2p,\alpha}^2e^{\sigma(t-t_0)}\rho\int_0^1\left((t-t_0)^\alpha\left\|\omega\right\|_{C(J;L^{2p}(\Omega))}+\eta C_{2p,\alpha}e^{\sigma(t-t_0)}\rho\right)^{p-2} \, d\eta\nonumber\\
    	&\quad (t-t_0)^{1-(p+1)\alpha}B(1-\alpha,1-p\alpha)\left\|v_1-v_2\right\|_{X_\alpha}\nonumber\\
	\le&~\sup_{t\in J}\left(\frac{\alpha}{e}\right)^\alpha\left\{1+\frac{C_\omega(t-t_0)^2}{(1-\alpha)(2-\alpha)}\right\}\nonumber\\
    	&\quad p(p-1)C_{2p,\alpha}^2e^{\sigma(t-t_0)}\rho\int_0^1\left((t-t_0)^\alpha\left\|\omega\right\|_{C(J;L^{2p}(\Omega))}+\eta C_{2p,\alpha}e^{\sigma(t-t_0)}\rho\right)^{p-2}d\eta\nonumber\\
    	&\quad (t-t_0)^{1-p\alpha}B(1-\alpha,1-p\alpha)\left\|v_1-v_2\right\|_{X_\alpha}\nonumber\\
	\le&\left(\frac{\alpha}{e}\right)^\alpha\left\{1+\frac{C_\omega\tau^2}{(1-\alpha)(2-\alpha)}\right\}p(p-1)C_{2p,\alpha}^2e^{\sigma\tau}\rho\nonumber\\
    	&\quad \left(\tau^\alpha\left\|\omega\right\|_{C(J;L^{2p}(\Omega))}+C_{2p,\alpha}e^{\sigma\tau}\rho\right)^{p-2}\tau^{1-p\alpha}B(1-\alpha,1-p\alpha)\left\|v_1-v_2\right\|_{X_\alpha}\nonumber\\
	=&~W(\tau)L_\omega(\rho)\rho\left\|v_1-v_2\right\|_{X_\alpha}.
\end{align*}
From the condition~\eqref{eqn:cond} stated in the theorem, $W(\tau)L_\omega(\rho)\rho<1$ holds.
Therefore, $\mathcal{S}$ becomes a contraction mapping under the assumptions of the theorem.
Banach's fixed-point theorem asserts that there exists the unique fixed point of $\mathcal{S}$ in $V$.
\end{proof}

\begin{remark}
The main feature of Theorem~\ref{thm:maintheorem} is that the sufficient condition~\eqref{eqn:cond} can be checked rigorously by verified numerical computations based on the interval arithmetic.
In addition, if the condition holds, the existence and local uniqueness of the mild solution are also proved in the sense of $C(J;L^2(\Omega))$ because there exists an embedding $D(\Delta_\mu^\alpha)\hookrightarrow L^2(\Omega)$.
For the detailed estimate of the residual in~\eqref{eqn:residual}, see Appendix A in \cite{bib:Mizuguchi2017}.
Moreover, when $1<p<2$, the conclusion of Lemma~\ref{lem:Lipschitz} does not hold.
Then a more careful estimate than Lemma~\ref{lem:Lipschitz} is necessary, as outlined in~\cite{bib:Tanaka2016}.
\end{remark}

\section{Concatenation Scheme}\label{sec:concatescheme}

After getting the local inclusion based on Theorem~\ref{thm:maintheorem}, we try to extend the time interval in which the mild solution is enclosed.
For this purpose, the initial function is replaced by a ball enclosing the mild solution at the endpoint,
we apply Theorem~\ref{thm:maintheorem} for the initial-boundary value problem on the next time interval, and repeat.
We called such a process a {\em concatenation scheme of verified numerical inclusion} in~\cite{bib:Mizuguchi2014}.

\subsection{Pointwise Error Estimate}\label{sec:shrinkingtech}
For a natural number $n$, let $0=t_0<t_1<\cdots<t_n<\infty$.
We denote $J_i=(t_{i-1},t_i]$ and $\tau_i=t_i-t_{i-1}$ ($i=1,2,\dots,n$).
We assume that the local inclusion of the mild solution is proved until the time interval $J_n$, i.e.,
the mild solution of \eqref{eqn:monJ} is locally enclosed in each $J_i$ so that
\[
	u\in B_{J_i}(\omega,\rho_i)=\left\{u:\sup_{t\in J_i}(t-t_{i-1})^\alpha e^{-\sigma_i(t-t_{i-1})}\|\Delta_\mu^\alpha(u-\omega)\|_{L^2}\le\rho_i\right\}
\]
for some $\sigma_i$, $\tau_i$, and $\rho_i$.
In this subsection, we give the error estimate at the endpoint of the time interval $J_n=(t_{n-1},t_n]$, namely,
we will describe how to obtain $\varepsilon_n$ satisfying $\|u(t_n)-\omega(t_n)\|_{L^2}\le\varepsilon_n$.
We call such an error estimate the {\em pointwise error estimate}.

For the scheme to succeed for a long time interval, the
pointwise error estimate should avoid the propagation of the previous estimate.
Such a propagation can cause dramatic over-estimation of the error.
This is called the {\em wrapping effect} in the field of verification methods for ordinary differential equations, and there are several techniques for avoiding it~\cite{bib:Berz1998,bib:kashikashi,bib:lohner,bib:Nedialkov1999,bib:Zgliczynski2002}, etc.
Similarly, some shrinking technique for the pointwise error estimate is necessary.
In the following, we will provide a technique for shrinking the propagation using another fixed-point formulation based on the evolution operator.

Define $B(t):=-\Delta-p\omega(t)^{p-1}$.
Since $A(t)$ generates the evolution operator and satisfies $A(t)=B(t)+\sigma_i$,
the real perturbed operator $B(t)$ also generates \cite{bib:Tanabe1960} the evolution operator $\left\{U_B(t,s)\right\}_{t_{i-1}\le s\le t\le t_i}$.
Letting $z=u-\omega$, the function $z$ for $t\in J_i$ satisfies
\[
	z(t)=U_B(t,t_{i-1})z(t_{i-1})+\int_{t_{i-1}}^tU_B(t,s)h_i(z(s)) \, ds,
\]
where
\[
	h_i\left(z\right)=\left(\omega+z\right)^p-\omega^p-p\omega^{p-1}z+\omega^p-\partial_{s} \omega+\Delta \omega.
\]
Furthermore, 
since $\lambda_A$ is the lower bound of the minimal eigenvalue of $A(t)$ for $t\in J_i$,
$\lambda_A-\sigma_i$ denotes the lower bound of the minimal eigenvalue of $B(t)$.
From Lemma \ref{lem:U_es}, it holds for $t$ and $s \in J_i$ that
\begin{equation}\label{eqn:UB_es}
	\|U_B(t,s)\varphi\|_{L^2}\le e^{-(\lambda_A-\sigma_i)(t-s)}\|\varphi\|_{L^2},~\forall\varphi\in L^2(\Omega).
\end{equation}

To obtain the pointwise estimate $\|u(t_n)-\omega(t_n)\|_{L^2}\le\varepsilon_n$,
we have the following estimates using \eqref{eqn:UB_es}:
\begin{align*}
	\|z(t_0)\|_{L^2}=\|u_0-\omega(t_0)\|_{L^2}\le\varepsilon_0,
\end{align*}
\begin{align*}
	\|z(t_1)\|_{L^2}\le&\left\|U_B(t_1,t_0)z(t_0)\right\|_{L^2}+\int_{J_{1}}\left\|U_B(t_1,s)h_1\left(z(s)\right)\right\|_{L^2} \, ds\\
	\le&e^{-(\lambda_A-\sigma_1)(t_1-t_0)}\left\|z(t_0)\right\|_{L^2}+\int_{J_{1}}\left\|U_B(t_1,s)h_1\left(z(s)\right)\right\|_{L^2} \, ds\\
	\le&e^{-(\lambda_A-\sigma_1)\tau_1}\varepsilon_0+\nu_1=:\varepsilon_1,
\end{align*}
where $\nu_i~(i=1,2,\dots,n)$ is the estimate satisfying
\begin{equation}\label{eqn:nu_i}
	\int_{J_{i}}\left\|U_B(t_i,s)h_i\left(z(s)\right)\right\|_{L^2} \, ds\le\nu_i.
\end{equation}
We repeat the pointwise estimates:
\begin{align}\label{eqn:z2}\nonumber
	\|z(t_2)\|_{L^2}\le&\left\|U_B(t_2,t_1)z(t_1)\right\|_{L^2}+\int_{J_{2}}\left\|U_B(t_2,s)h_2\left(z(s)\right)\right\|_{L^2} \, ds\\
	\le&e^{-(\lambda_A-\sigma_2)(t_2-t_1)}\left\|z(t_1)\right\|_{L^2}+\int_{J_{2}}\left\|U_B(t_2,s)h_2\left(z(s)\right)\right\|_{L^2} \, ds\nonumber\\
	\le&e^{-(\lambda_A-\sigma_2)\tau_2}\left(e^{-(\lambda_A-\sigma_1)\tau_1}\varepsilon_0+\nu_1\right)+\nu_2\nonumber\\
	=&\left(e^{-(\lambda_A-\sigma_2)\tau_2}e^{-(\lambda_A-\sigma_1)\tau_1}\right)\varepsilon_0+e^{-(\lambda_A-\sigma_2)\tau_2}\nu_1+\nu_2=:\varepsilon_2.
\end{align}
In the last term of \eqref{eqn:z2}, the previous error estimate $\varepsilon_1$ does not appear.
Instead, $\varepsilon_0$, $\nu_1$, and $\nu_2$ are used for obtaining $\varepsilon_2$.
This is the essential point for avoiding the propagation of the previous estimate because we obtain the pointwise error estimate without the previous error estimate $\varepsilon_1$.
In the same way, we have
\begin{align*}
	\|z(t_3)\|_{L^2}\le&\left\|U_B(t_3,t_2)z(t_2)\right\|_{L^2}+\int_{J_{3}}\left\|U_B(t_3,s)h_3\left(z(s)\right)\right\|_{L^2} \, ds\\
	\le&e^{-(\lambda_A-\sigma_3)(t_3-t_2)}\left\|z(t_2)\right\|_{L^2}+\int_{J_{3}}\left\|U_B(t_3,s)h_3\left(z(s)\right)\right\|_{L^2} \, ds\\
	\le&e^{-(\lambda_A-\sigma_3)\tau_3}\left(\left(e^{-(\lambda_A-\sigma_2)\tau_2}e^{-(\lambda_A-\sigma_1)\tau_1}\right)\varepsilon_0+e^{-(\lambda_A-\sigma_2)\tau_2}\nu_1+\nu_2\right)+\nu_3\\
	=&\left(e^{-(\lambda_A-\sigma_3)\tau_3}e^{-(\lambda_A-\sigma_2)\tau_2}e^{-(\lambda_A-\sigma_1)\tau_1}\right)\varepsilon_0+\left(e^{-(\lambda_A-\sigma_3)\tau_3}e^{-(\lambda_A-\sigma_2)\tau_2}\right)\nu_1\\
	&+e^{-(\lambda_A-\sigma_3)\tau_3}\nu_2+\nu_3=:\varepsilon_3
\end{align*}
and, consequently, the desired estimate is
\begin{align}\label{eqn:zn}\nonumber
	\|z(t_n)\|_{L^2}\le&\left(e^{-(\lambda_A-\sigma_n)\tau_n}e^{-(\lambda_A-\sigma_{n-1})\tau_{n-1}}\cdots e^{-(\lambda_A-\sigma_1)\tau_1}\right)\varepsilon_0\\
	&+\left(e^{-(\lambda_A-\sigma_n)\tau_n}e^{-(\lambda_A-\sigma_{n-1})\tau_{n-1}}\cdots e^{-(\lambda_A-\sigma_2)\tau_2}\right)\nu_1\nonumber\\
	&+\left(e^{-(\lambda_A-\sigma_n)\tau_n}e^{-(\lambda_A-\sigma_{n-1})\tau_{n-1}}\cdots e^{-(\lambda_A-\sigma_3)\tau_3}\right)\nu_2\nonumber\\
	&+\cdots+e^{-(\lambda_A-\sigma_n)\tau_n}\nu_{n-1}+\nu_n=:\varepsilon_n.
\end{align}
By handling the inside of {all of the parentheses} in~\eqref{eqn:zn} first, we expect the error estimate to avoid the propagation of previous estimates.
These estimates are imitations of techniques for avoiding the wrapping effect in verification methods for ordinary differential equations~\cite{bib:kashikashi,bib:lohner,bib:Zgliczynski2002}.
In the actual computation, we first store the estimates $\varepsilon_0$ and $\nu_i$.
After that, we multiply these by the insides of the parentheses.
We then obtain the pointwise error estimate without the previous error estimate.
In Section \ref{sec:numericalexam}, we illustrate the efficiency of the proposed shrinking technique by numerical examples.

\begin{remark}
The estimate \eqref{eqn:zn} is always satisfied regardless of the positiveness or negativeness of $\lambda_A-\sigma_i$.
If $\lambda_A-\sigma_i<0$, the term  $e^{-(\lambda_A-\sigma_{i})\tau_{i}}$ increases the error estimate  due to $e^{-(\lambda_A-\sigma_{i})\tau_{i}}>1$.
In that case, to obtain an accurate error estimate, one should take sufficiently small $\tau_{i}$ so that $e^{-(\lambda_A-\sigma_{i})\tau_{i}}\approx1$ holds.
If we take a large step size, the error estimate may increase rapidly.
On the other hand, for a small step size, it may be difficult to continue the concatenation scheme for a long time.
Therefore, there is a trade-off between the length of the step size and the accuracy of the error estimate.
\end{remark}

\subsection{Method of Estimating $\nu_i$}
In this subsection, we introduce how to estimate $\nu_i$ in \eqref{eqn:nu_i}.
Setting $z_1=z$, and $z_2=0$ in Lemma~\ref{lem:Lipschitz}, we have the following estimate for a fixed $s\in J_i$:
\begin{align*}
	&\hspace{-30pt}\left\|h_i\left(z(s)\right)\right\|_{L^2}\\
	\le&~\left\|(\omega+z)^p-\omega^p-p\omega^{p-1}z\right\|_{L^2}+\left\|\omega^p-\partial_s\omega+\Delta\omega\right\|_{L^2}\\
	\le&~p(p-1)C_{2p,\alpha}^2\left\|\Delta_\mu^\alpha z\right\|_{L^2}^2\int_0^1\int_0^1\left(\|\omega\|_{L^{2p}}+\theta\eta C_{2p,\alpha}\left\|\Delta_\mu^\alpha z\right\|_{L^2}\right)^{p-2} \, d\eta\theta d\theta\\
	    &\quad+\left\|\omega^p-\partial_t\omega+\Delta\omega\right\|_{L^2}\\
	=&~p(p-1)C_{2p,\alpha}^2(s-t_{i-1})^{-p\alpha}e^{p\sigma_i(s-t_{i-1})}\left\|(s-t_{i-1})^\alpha e^{-\sigma_i(s-t_{i-1})}\Delta_\mu^\alpha z\right\|_{L^2}^2\\
	    &\quad \int_0^1\int_0^1\left\{(s-t_{i-1})^\alpha e^{-\sigma_i(s-t_{i-1})}\|\omega\|_{L^{2p}} \right. \\
	    & \qquad\qquad \left. +\theta\eta C_{2p,\alpha}\left\|(s-t_{i-1})^\alpha e^{-\sigma_i(s-t_{i-1})}\Delta_\mu^\alpha z\right\|_{L^2}\right\}^{p-2} \, d\eta\theta \, d\theta\\
	    &\quad +\left\|\omega^p-\partial_t\omega+\Delta\omega\right\|_{L^2}\\
	\le&~p(p-1)C_{2p,\alpha}^2(s-t_{i-1})^{-p\alpha}e^{p\sigma_i(s-t_{i-1})}\rho_i^2 \\
	&\quad \int_0^1\int_0^1\left\{(s-t_{i-1})^\alpha e^{-\sigma_i(s-t_{i-1})}\|\omega\|_{L^{2p}}+\theta\eta C_{2p,\alpha}\rho_i\right\}^{p-2} \, d\eta\theta \, d\theta+\delta_i,
\end{align*}
using the estimate $\left\|\omega^p-\partial_t\omega+\Delta\omega\right\|_{C(J_i;L^2(\Omega))}\le\delta_i$.
Therefore, the estimate \eqref{eqn:nu_i} is
\begin{align*}
	&\hspace{-20pt}\int_{J_{i}}\left\|U_B(t_i,s)h_i\left(z(s)\right)\right\|_{L^2} \, ds\\
	\le&~\int_{J_{i}}e^{-(\lambda_A-\sigma_i)(t_i-s)}\Bigg(p(p-1)C_{2p,\alpha}^2(s-t_{i-1})^{-p\alpha}e^{p\sigma_i(s-t_{i-1})}\rho_i^2\\
	    &\quad \int_0^1\int_0^1\left\{(s-t_{i-1})^\alpha e^{-\sigma_i(s-t_{i-1})}\|\omega\|_{L^{2p}}+\theta\eta C_{2p,\alpha}\rho_i\right\}^{p-2} \, d\eta\theta  \, d\theta+\delta_i\Bigg) \, ds\\
	\le&~p(p-1)C_{2p,\alpha}^2\rho_i^2
	\int_0^1\int_0^1\left\{(t_i-t_{i-1})^\alpha\|\omega\|_{C(J_i;L^{2p}(\Omega))}+\theta\eta C_{2p,\alpha}\rho_i\right\}^{p-2} \, d\eta\theta  \, d\theta\\
	    &\quad \int_{J_{i}}e^{-(\lambda_A-\sigma_i)(t_i-s)}(s-t_{i-1})^{-p\alpha}e^{p\sigma_i(s-t_{i-1})}ds+\delta_i\int_{J_{i}}e^{-(\lambda_A-\sigma_i)(t_i-s)} \, ds\\
	=&~p(p-1)C_{2p,\alpha}^2\rho_i^2
	\int_0^1\int_0^1\left\{\tau_i^\alpha\|\omega\|_{C(J_i;L^{2p}(\Omega))}+\theta\eta C_{2p,\alpha}\rho_i\right\}^{p-2} \, d\eta\theta  \, d\theta\\
	    &\quad \int_{J_{i}}e^{-(\lambda_A-\sigma_i)(t_i-s)}(s-t_{i-1})^{-p\alpha}e^{p\sigma_i(s-t_{i-1})} \, ds+\delta_i\left(\frac{1-e^{-(\lambda_A-\sigma_i)\tau_i}}{\lambda_A-\sigma_i}\right),
\end{align*}
where we have two cases for the estimate
\begin{align*}\label{eqn:kaizen}\nonumber
	&\hspace{-30pt}\int_{J_{i}}e^{-(\lambda_A-\sigma_i)(t_i-s)}(s-t_{i-1})^{-p\alpha}e^{p\sigma_i(s-t_{i-1})} \, ds\\[2mm]
	&\le\left\{
	\begin{array}{ll}
		\displaystyle \frac{e^{p\sigma_i\tau_i}\tau_i^{1-p\alpha}}{1-p\alpha}&\lambda_A-\sigma_i\ge0,\\[4mm]
		\displaystyle \frac{e^{((p+1)\sigma_i-\lambda_A)\tau_i}\tau_i^{1-p\alpha}}{1-p\alpha}&\lambda_A-\sigma_i<0.
	\end{array}\right.
\end{align*}
Then, if $\lambda_A-\sigma_i\ge0$ holds, we  take
\begin{align*}
	\nu_i=& ~p(p-1)C_{2p,\alpha}^2\rho_i^2
	\int_0^1\int_0^1\left\{\tau_i^\alpha\|\omega\|_{C(J_i;L^{2p}(\Omega))}+\theta\eta C_{2p,\alpha}\rho_i\right\}^{p-2} \, d\eta\theta \, d\theta \\
	& \quad \left(\frac{e^{p\sigma_i\tau_i}\tau_i^{1-p\alpha}}{1-p\alpha}\right)
	+\delta_i\left(\frac{1-e^{-(\lambda_A-\sigma_i)\tau_i}}{\lambda_A-\sigma_i}\right).
\end{align*}
Otherwise, in the case of $\lambda_A-\sigma_i<0$, we take
\begin{align*}
	\nu_i=& ~p(p-1)C_{2p,\alpha}^2\rho_i^2
	\int_0^1\int_0^1\left\{\tau_i^\alpha\|\omega\|_{C(J_i;L^{2p}(\Omega))}+\theta\eta C_{2p,\alpha}\rho_i\right\}^{p-2} \, d\eta\theta \, d\theta \\
	& \quad \left(\frac{e^{((p+1)\sigma_i-\lambda_A)\tau_i}\tau_i^{1-p\alpha}}{1-p\alpha}\right)
	+\delta_i\left(\frac{e^{(\sigma_i-\lambda_A)\tau_i}-1}{\sigma_i-\lambda_A}\right).
\end{align*}

\begin{remark}
Controlling $\tau_i$ is also necessary so that $\nu_i$ is as small as possible.
For example, we should take sufficiently small $\nu_i$ so that $e^{p\sigma_i\tau_i}\approx 1$ holds for the case of $\lambda_A-\sigma_i>0$ and that $e^{((p+1)\sigma_i-\lambda_A)\tau_i}\approx 1$ holds for the case of $\lambda_A-\sigma_i<0$.
\end{remark}

\section{Numerical Examples}\label{sec:numericalexam}

To illustrate the efficiency of our verification method, we show some numerical results.
Let $\Omega:=(0,1)^2$ be a unit square domain\footnote{It {can be proved} that $\lambda_{\min}=2\pi^2$ on $\Omega$.} in $\mathbb{R}^2$, and set $p=2$ in~\eqref{eqn:monJ} to
consider the semilinear parabolic equation so called the {\em Fujita-type} equation
\begin{equation}\label{eqn:fujitaeq}
\left\{
\begin{array}{ll}
	\partial_tu-\Delta u=u^2 & \mbox{in}~(0,T)\times\Omega,\\[1mm]
	u(t,x)=0,&t\in (0,T),~x\in\partial{\Omega},\\[1mm]
	u(t_0,x)=u_0(x),&x\in\Omega,
\end{array}
\right.
\end{equation}
where $T>0$ is a fixed time, $u_0(x)=\gamma\sin(\pi x_1)\sin(\pi x_2)$, and $\gamma$ is a parameter.
It is well-known that, for a sufficiently large $\gamma$, the solution of \eqref{eqn:fujitaeq} blows up\footnote{When $\gamma=27$, the approximate solution seems to blow up in finite time.} in finite time \cite{bib:Ferreira2012}.
When the scale of an approximate solution becomes large, it is difficult to verify the existence and local uniqueness of the exact solution using the verification method.
Hence, we consider that this problem is suite for a benchmark test for illustrating the accuracy of verification methods.

All computations are carried out on CentOS 6.3, Intel(R) Xeon(R) CPU E5-2687W@3.10 GHz, and MATLAB 2016b with INTLAB \cite{bib:intlab} version 9.
The approximate solution $\omega$ is
\[
	\omega(t,x) = \sum_{|m|\le N^2}\tilde{u}_m(t)\psi_m(x),
\]
where $m=(m_1,m_2)$ is a multi-index, and $\psi_m(x)=\sin(m_1\pi x_1)\sin(m_2\pi x_2)$ ($m_1,m_2=1,2,\dots,N$).
We define the finite dimensional subspace of $D(\Delta)$ as
\[
	V_N:=\left\{\sum_{|m|\le N^2}a_{m}\psi_m(x)~:~ a_{m}\in\mathbb{R}\right\}.
\]
Each $\tilde{u}_m(t)$ of $\omega$ is given by the Fourier-Galerkin procedure {\cite{bib:Canuto1988}} and the Crank-Nicolson scheme {\cite{bib:fujita2001,bib:thomee2013}} in the time variable.
Namely, for $0=t_0<t_1<\cdots<t_n=T$, we employ the full discretization scheme for obtaining $\left\{u_i\right\}_{i\ge1}\subset V_N$ such that
\[
	\left(\frac{u_i-u_{i-1}}{\tau_i},\phi_N\right)_{L^2}-\frac{1}{2}\left(\Delta(u_i+u_{i-1}),\phi_N\right)_{L^2}=\frac{1}{2}\left(u_i^2+u_{i-1}^2,\phi_N\right)_{L^2},~~\forall \phi_N\in V_N.
\]
Let $\hat u_i\in V_N$ ($i=0,1,\dots,n$) be the numerical approximation of $u_i$.
The approximate solution $\omega$ is constructed of $\omega(t,x)=\sum_{i=0}^n\hat u_i(x)l_i(t)$, where $l_i(t)$ is the linear Lagrange basis {\cite{bib:Kress1998}} satisfying $l_i(t_j)=\delta_{ij}$ for $j=0,1,\dots,n$ (Kronecker's delta).
Furthermore, because each $\hat u_i(x)$ is described by $\hat u_i(x)=\sum_{|m|\le N^2}\tilde u_{i,m}\psi_m(x)$, 
\[
	\tilde{u}_m(t):=\sum_{i=0}^n\tilde u_{i,m}l_i(t).
\]

\begin{figure}[ht]
\centering
\includegraphics[width=\textwidth]{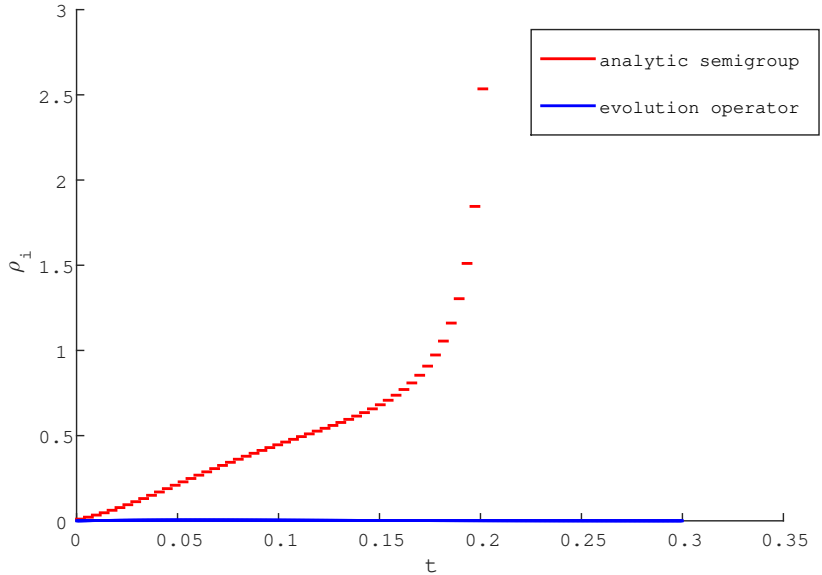}
\caption{Comparison with the previous verification method \cite{bib:Mizuguchi2017} using the analytic semigroup; $\gamma=6.8$, $\alpha=1/2$ (analytic semigroup), $\alpha=3/8$ (evolution operator), $\mu=70$, and $N=5$.
Each $\rho_i$ is plotted versus $t$. We stopped the concatenation scheme using the evolution operator at $T=0.3$.}
\label{fig:fig1}
\end{figure}

The first result is compared with the previous verification method~\cite{bib:Mizuguchi2017} using the analytic semigroup generated by $\Delta$.
Figure \ref{fig:fig1} displays each radius of $B_{J_i}(\omega,\rho_i)$ in which the mild solution of \eqref{eqn:fujitaeq} is locally enclosed.
The result of inclusion using the evolution operator is more accurate than that of the previous one (using analytical semigroup) because the concatenation scheme succeeds in enclosing the mild solution for a long time.
In particular, for the previous verification method, the accumulation of the error estimate causes the failure in enclosing the mild solution at $t=0.203125$.
On the other hand, the concatenation scheme based on the evolution operator can continue the numerical verification without the accumulation of the error estimate.

\begin{figure}[ht]
\centering
\includegraphics[width=\textwidth]{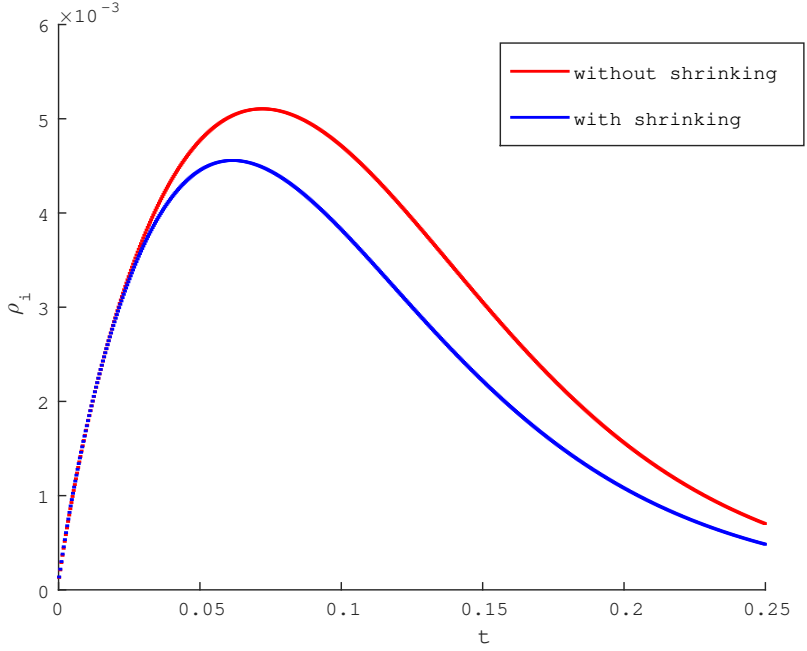}
\caption{The efficiency of the shrinking technique; $\gamma=7$, $\alpha=3/8$, $\mu=70$, and $N=5$. Each $\rho_i$ is plotted versus $t$. The concatenation scheme is stopped at $T=0.25$.}
\label{fig:meritofshrinking1}
\end{figure}

\begin{figure}[ht]
\centering
\includegraphics[width=\textwidth]{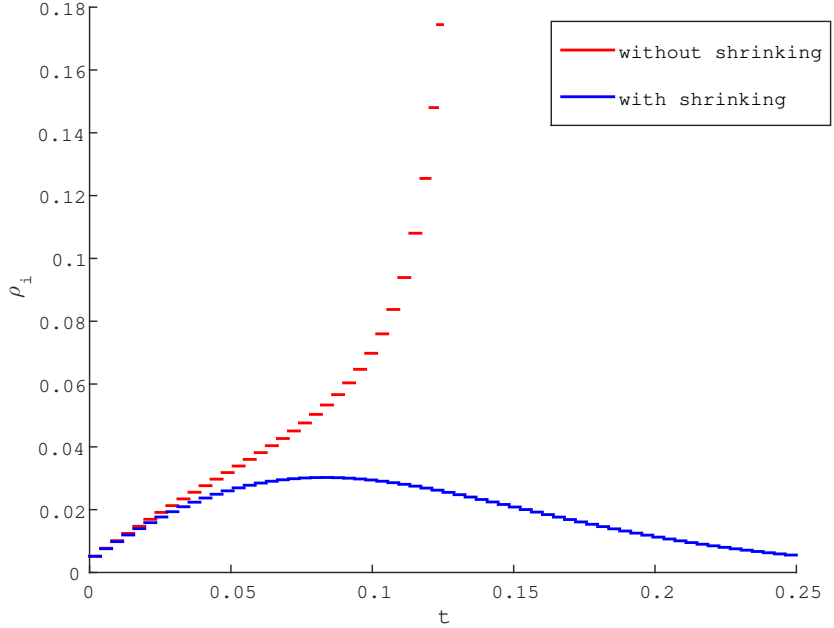}
\caption{The efficiency of our shrinking technique by taking a rough step size; $\gamma = 7$, $\alpha = 3/8$, $\mu = 70$, and $N = 5$. Each $\rho_i$ is plotted versus $t$. In the result without shrinking, the accumulation of the error estimate occurs.  The concatenation scheme with our shrinking technique is stopped at $T = 0.25$.
}
\label{fig:meritofshrinking2}
\end{figure}

Next, we illustrate the efficiency of the shrinking technique using the pointwise error estimate discussed in Section~\ref{sec:shrinkingtech}.
Figures~\ref{fig:meritofshrinking1} and~\ref{fig:meritofshrinking2} show that the shrinking technique can control the propagation of the previous estimate to some extent.
In Figure~\ref{fig:meritofshrinking1}, the error estimate $\rho_i$ using our shrinking technique is slightly larger than that without using shrinking technique for first several steps.
After that, the estimate becomes tighter than that without using our shrinking technique.
This implies that the shrinking technique reduces propagation of previous excess widths.
Furthermore, in Figure \ref{fig:meritofshrinking2}, if we take a rough step size, the propagation of the previous estimate causes failure in enclosing the mild solution at $t=0.12482732616559$.
Such a failure does not occur in the result with the shrinking technique,
demonstrating the effectiveness of the shrinking technique.
As a result, the concatenation scheme succeeds in enclosing the mild solution for a long time.

\begin{figure}[ht]
\centering
\includegraphics[width=\textwidth]{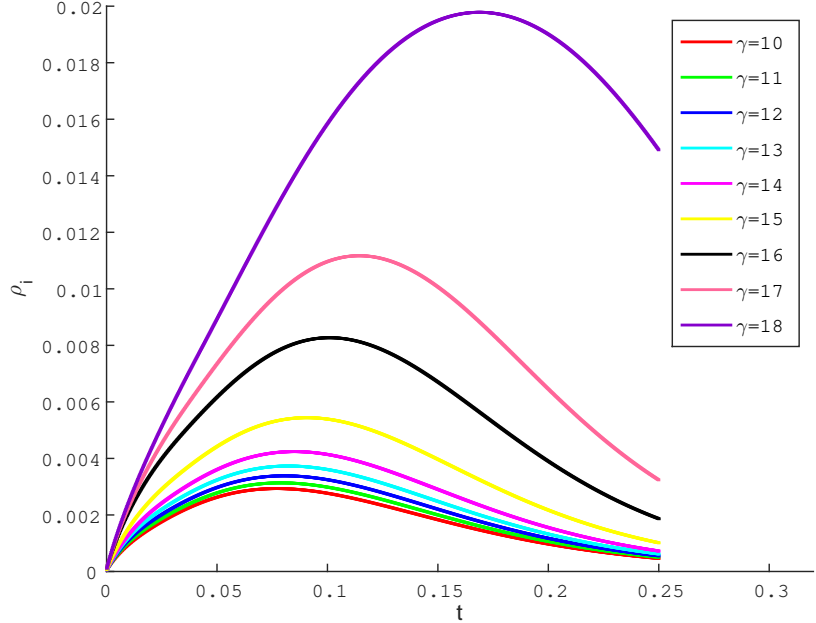}
\caption{Results of the concatenation scheme varying $\gamma = 10, 11, \ldots, 18$ ($\alpha = 3/8$, $\mu = 600$\,{\footnotesize $(\gamma < 15)$}, $550$\,{\footnotesize $(\gamma = 15)$}, $350$\,{\footnotesize $(\gamma > 15)$},~$N = 11$). Each $\rho_i$ is plotted versus $t$. The concatenation scheme succeeds in enclosing the mild solution of~\eqref{eqn:fujitaeq} until $T = 0.25$.}
\label{fig:fig4}
\end{figure}

Finally, we display the results of the concatenation scheme varying $\gamma = 10$, $\ldots$, $18$ in Figure~\ref{fig:fig4}.
For each $\gamma$, the concatenation scheme succeeds in enclosing the mild solution of~\eqref{eqn:fujitaeq} until at least $T = 0.25$.
In this example, the choice of the shift value $\mu$ is the key to the success of numerical verification.
For example, if we set $\mu = 600$ in the case of $\gamma = 15$, the verification cannot succeed as far as $T = 0.25$.
In such a case, the error estimate accumulates, and the condition of the local inclusion is not satisfied.
This implies that there exists an optimal shift value $\mu$ that depends probably on both $\gamma$ and $N$.
In the actual computation, we experiment to find the appropriate value of $\mu$.
Furthermore, when $\gamma = 19$, the verification scheme cannot enclose the mild solution as far as $T = 0.25$.
In this case, a more accurate approximate solution is necessary, e.g., increasing the number of bases to $N = 15$ or more.
On the other hand, more computational resources are needed for the numerical verification using such an accurate approximate solution.

\section{Conclusion}

In this paper, we have discussed an accurate method for guaranteeing the existence and local uniqueness of mild solutions of semilinear heat equations.
Our method consists of a fixed-point formulation using the {\em evolution operator} instead of the analytic semigroup; the {\em pointwise error estimate} by rearranged computing for shrinking the propagation of over-estimates; and the {\em concatenation scheme} to extend the time interval in which the mild solution is enclosed.
As a result, compared with the previous result (using the analytic semigroup), our method can enclose the mild solution for a long time.
We have also provided numerical results to illustrate the efficiency of our verification method. 

We conclude this paper by commenting on potential extensions.
Our method could be extended to more general nonlinearity when $f(u)$ is a mapping from $\mathbb{R}$ to $\mathbb{R}$ such that $f\circ u\in L^2(\Omega)$ holds for each $u\in H^1_0(\Omega)$.
In addition, we assume that 
\begin{itemize}
\item the mapping $f$ is Fr\'echet differentiable {(see, e.g, \cite{bib:Muscat2014})} in the sense that $f$ is an operator from $H^1_0(\Omega)$ to $L^2(\Omega)$, and
\item there exists a monotonically non-decreasing function $L_{\omega}:[0,\infty)\to [0,\infty)$ corresponding to the first Fr\'echet derivative {\cite{bib:Muscat2014}} of $f$ such that
\[
	\left\|(f'[\omega+z]-f'[\omega])\phi\right\|_{L^2}\le L_{\omega}\left(\rho\right)\rho\|\Delta_\mu^\alpha\phi\|_{L^2}
\]
for any $\phi\in D\left(\Delta_\mu^\alpha\right)$, each $t\in J$, and
$z\in B_J(0,\rho)$ for a certain $\rho>0$ defined in Section \ref{sec:localinc}.
\end{itemize}
Furthermore, the space domain $\Omega$ could be generalized to a convex polyhedral domain in $\mathbb{R}^d$ ($d=1,2,3$) by using a finite element method.

\section*{Acknowledgements}
The authors express their sincere gratitude to Prof. M. Kashiwagi in Waseda University for his valuable comments and kind remarks.
In particular, his essential suggestion of the shrinking technique for the pointwise estimate was beneficial.
The authors would also like to express their gratitude to the two anonymous referees for providing comments that improve this paper.

\bibliographystyle{plain}
\bibliography{evolutionop}
\end{document}